\documentclass{article}

\usepackage{amsmath}
\usepackage{amssymb}
\usepackage{dsfont}
\usepackage{bm}
\usepackage{enumitem}
\usepackage{verbatim}
\usepackage{amsthm}
\usepackage{tikz}
\usepackage{subcaption} 
\usepackage{xcolor}

\usepackage[english]{babel}
\usepackage[utf8]{inputenc}

\numberwithin{equation}{section}

\theoremstyle{plain} 
\newtheorem{thm}{Theorem}[section]
\newtheorem{seur}[thm]{Corollary}
\newtheorem{lem}[thm]{Lemma}

\theoremstyle{remark}
\newtheorem{exa}[thm]{Example}
\newtheorem{defn}[thm]{Definition}
\newtheorem{rem}[thm]{Remark}
\newtheorem{algo}[thm]{Algorithm}

\counterwithin{figure}{section}

\newcommand{\X}{\textbf{X}}

\newcommand{\Uv}{\textbf{U}}
\newcommand{\uv}{\textbf{u}}
\newcommand{\x}{\textbf{x}}
\newcommand{\y}{\textbf{y}}
\newcommand{\z}{\textbf{z}}
\newcommand{\vv}{\textbf{v}}
\newcommand{\av}{\textbf{a}}
\newcommand{\bv}{\textbf{b}}

\newcommand{\RR}{\mathbb R}

\newcommand{\Ss}{\mathbb S}

\newcommand{\PP}{\mathbb P}
\newcommand{\EE}{\mathbb E}
\newcommand{\ind}{\mathds{1}}

\newcommand{\dist}{\mathrm{dist}}

\begin{document}

\title{On the identification of the riskiest directional components from multivariate heavy-tailed data}
\author{Miriam Hägele and Jaakko Lehtomaa}

\maketitle

\begin{abstract}
    In univariate data, there exist standard procedures for identifying dominating features that produce the largest observations. However, in the multivariate setting, the situation is quite different. This paper aims to provide tools and methods for detecting dominating directional components in multivariate data.
    
    We study general heavy-tailed multivariate random vectors in dimension $d\geq 2$ and present procedures which can be used to evaluate why the data is heavy-tailed. This is done by identifying the set of the riskiest directional components. The results are of particular interest in insurance when setting reinsurance policies and in finance when hedging a portfolio of multiple assets.
\end{abstract}


{\bf Classification:}
60E05, 91B30, 91B28, 62P05

\section{Introduction}

It is not uncommon to find heavy-tailed features in multivariate data sets in insurance and finance \cite{Embrechts1,Peters1}. Since financial entities seek ways to reduce the total risks of their portfolios, it is necessary to understand what the main sources of risks are. Once this is known, one can seek optimal ways of reducing riskiness. In insurance, where the multivariate observations could consist of losses of different lines of business, the companies are typically interested in finding the best suited reinsurance policy. In finance, where the data could consist of returns of multiple assets, the aim could be to find an optimal hedging strategy against large losses. Even though multivariate heavy-tailed distributions are encountered frequently in various applications, there does not exist a general framework for analysis. 

In earlier research, several models for what a heavy-tailed random vector should mean have been introduced \cite{Cline1, Omey1, Resnick1,Samorodnitsky}. One of the simplest ways to model the situation is to present a $d$-dimensional random vector $\X$ in polar coordinates using the length $R$ of the vector and the directional vector $\Uv$ on the unit sphere $\Ss^{d-1}$ so that $\X=R\Uv$. Here, perhaps the simplest assumption would be that all directions are equally likely, i.e.\ the random vector $\Uv$ has a uniform distribution on the unit sphere. In realistic models, some directions are more likely than others specifically when $R$ is large. Efforts have been made to capture the heterogeneity in the distribution of $\Uv$ by studying, for example, the set of elliptic distributions \cite{Hult5, Kluppelberg2, Li2}. However, there are typical data sets that do not fit well to these assumptions because the support of $\Uv$ is concentrated into a cone and does not cover the entire set $\Ss^{d-1}$. There appears to be a need for models that allow the heaviness of the tail to vary in different directions as in \cite{Weng1} but which are not restricted by a parametric class of distributions. Ideally, the model should admit the tail of $R$ to have more general form than, say, only the form of a power function.

To answer the question on what causes the heavy-tailedness of multivariate data, one usually selects a suitable norm $\|\cdot\|$ and analyses the resulting one-dimensional distribution of $\|\X\|$ with standard procedures, such as the mean excess plot \cite{Das3, Ghosh1, Ghosh2} or alternative methods such as the ones in \cite{Asmussen3}. In practice, given the heaviness identified from the distribution of $\|\X\|$, the aim is to analyse the distribution of $\Uv$ and the dependence structure between $R$ and $\Uv$. If the conditional distribution exists, we can write
\begin{equation}\label{eq_taildef}
    \PP(R>k \ |\ \Uv=\x)= e^{-h(k,\x)}, \quad k\geq 0,
\end{equation}
where $h(\cdot,\x)$ is an increasing function for a fixed $\x \in \Ss^{d-1}$. In this setting, the problem is to identify which set of vectors $\x \in \Ss^{d-1}$ produces the heaviest or dominating conditional tails of the form \eqref{eq_taildef}. The dominating directions are the vectors $\x \in \Ss^{d-1}$ for which the function $h(k,\x)$ grows at the slowest rate, as $k\to \infty$. Presentation \eqref{eq_taildef} admits the study of a wide class of distributions.

We derive a presentation for the set of directional components that produce the heaviest tails and present a procedure to be used in practical analysis. The method applied to generated data returns a subset of the space which gives more information about the distribution than a single number or vector, but demands less data than a full empirical measure. Furthermore, the original data does not have to be pre-processed or transformed and the method is also applicable to data, where in some directions, the observations are sparse or observations do not occur at all. This is different from the grid-based approaches such as \cite{Lehtomaa3} where the entire space is first divided into cells and each cell is studied separately. In such approaches, there can exist empty or sparsely populated cells which might be problematic for analysis. The presented methods can be applied, in particular, to analyse financial data, where extremal observations typically appear in two opposite directions, which might be unknown.

\subsection*{Notation}
For a set $A$, we denote its closure by $\overline{A}$ or $\textrm{cl}(A)$, its interior by $\textrm{int}(A)$, its complement by $A^c$ and $\emptyset$ denotes an empty set. The symbol $A\backslash D$ means the set $A\cap D^c$. The term $A\subset D$ means $A$ is a subset of $D$ whereas $A\subsetneq D$ implies that $A$ is a proper subset of $D$. The notation $a(n)=o(b(n))$ refers to the little o-notation and means $a(n)/b(n)\to 0$, as $n\to\infty$. We use the convention $\log(0)=-\infty$.

\section{Assumptions and Definitions}
We write a $d$-dimensional random vector $\X$, where $d\geq 2$, in the form
\begin{equation}\label{eq_xdef}
 \X= R\Uv,
\end{equation}
where $R=\|\X\|$ and $\Uv=\frac{\X}{\|\X\|}$.
Here, we assume that the norm is an $l_p$ norm with $1\leq p<\infty$. The unit sphere is the set $\Ss^{d-1}=\{\x \in \RR^{d} \colon ||\x||=1\}$. Then, we have a geodesic metric on $\Ss^{d-1}$, defined in detail in terms of the geodesic distance in Section \ref{sec_geodesic}. In particular, the space $\Ss^{d-1}$ equipped with the geodesic metric is a complete metric space. Open sets, balls and other topological concepts on $\Ss^{d-1}$ are defined using this metric. For example, $B(\x,\varepsilon)\subset \Ss^{d-1}$ is an open ball with centre $\x$ and radius $\varepsilon>0$. The Borel sigma-algebra, the sigma-algebra generated by open sets, is denoted by $\mathcal{B}$, e.g.\ $\mathcal{B}(\Ss^{d-1})$. In the case of spherical or elliptical distributions, we restrict the norm to be the $l_2$-norm denoted by $\|\cdot\|_2$. This restriction ensures that ellipsoids have their natural interpretation. 

\subsection{Assumptions}
We need the following conditions in the formulation of the results. The required conditions are indicated in the assumptions of each result.
\begin{enumerate}[label=(A\arabic*)]
    \item \label{as1}
    $R$ is a positive random variable with right-unbounded support. For any $k\in \RR$, it holds that $\PP(R>k)>0$ and $R$ is heavy-tailed in the sense that $\lim_{k\to \infty} -\log(\PP(R>k))/k=0$.
    \item \label{as2} 
    $\Uv$ is a random vector on the unit sphere $\Ss^{d-1}$ such that the quantity $$\PP(\Uv\in A|R>k)$$
    remains constant for all $k>k_0$ where $k_0>0$ is a fixed number that does not depend on the Borel set $A \subset \Ss^{d-1}$. In particular, the limiting probability distribution of $\Uv\, |\, R>k$ exists, as $k\to\infty$. In fact, the limiting distribution on $\Ss^{d-1}$ is reached once $k>k_0$.
    \item \label{as3} 
    The limit $\lim_{k\to \infty}g(k,A)$ exists in $[0,\infty]$ for all Borel sets $A \subset \Ss^{d-1}$, where $g:(k_R,\infty)\times \mathcal{B}(\Ss^{d-1}) \to \RR$ is defined as
    \begin{equation}\label{eq_gdef}
    g(k,A) = \frac{\log(\PP(R>k,\Uv\in A))}{\log(\PP(R>k))}
    \end{equation}
    and $k_R=\sup\{k:\PP(R>k)=1\}.$
\end{enumerate}

\begin{rem}
The condition on heavy-tailedness in Assumption \ref{as1} is equivalent with the condition $\EE(e^{sR})=\infty$ for all $s>0$ which is the usual definition of a (right) heavy-tailed real-valued random variable $R$. The distribution of $\Uv$ does not have to be uniform on the unit sphere $\Ss^{d-1}$. In fact, the distribution of $\Uv$ does not even have to have probability mass in every direction. 
\end{rem}

The risk function of $R=\|\X\|$ is defined as 
\begin{displaymath}
h(k)=-\log(\PP(R>k)).
\end{displaymath}
The function $h$ provides a benchmark against which risk functions calculated from subsets can be compared. In this sense, the function $h$ indicates what the heaviness of the tail will be in the set of the riskiest directions $S$. The logarithmic transformation is used widely in asymptotic analysis. One reason for the use of this particular transformation of the tail function instead of some other transformation is that, roughly speaking, the function 
$$x\mapsto -\log \PP(R>x)$$ is concave for heavy-tailed $R$ and convex for light-tailed $R$. Assuming \ref{as1} indicates that we operate in the heavy-tailed regime.

Here, $S$ is the set
\begin{align}
    \Bigg\{\x \in \Ss^{d-1} \colon  & \lim_{k\to \infty} \frac{\log(\PP(R>k|\Uv\in  B(\x,\varepsilon)))}{\log(\PP(R>k))}=1, 
    \nonumber \\ 
    &\PP(\Uv\in B(\x,\varepsilon))>0,\,  \forall \varepsilon>0\Bigg\}.\label{eq_domset}
\end{align}

The limit in \eqref{eq_domset} can be written in equivalent forms.

\begin{lem}\label{lem_S}
Suppose \ref{as1}-\ref{as3} hold and $\PP(\Uv\in B(\x,\varepsilon))>0$. Then, 
\begin{equation}\label{eq_sconnectionlemma_eq1}
\lim_{k\to \infty}\frac{\log(\PP(R>k|\Uv\in  B(\x,\varepsilon)))}{\log(\PP(R>k))}=\lim_{k\to \infty} \frac{\log(\PP(R>k,\Uv\in  B(\x,\varepsilon)))}{\log(\PP(R>k))}.
\end{equation}
In addition, \eqref{eq_sconnectionlemma_eq1} can be written as
\begin{equation}\label{eq_sconnectionlemma_eq2}
\lim_{k\to \infty} \frac{\log(\PP(\Uv\in  B(\x,\varepsilon)|R>k))}{\log(\PP(R>k))}-1.
\end{equation}
\end{lem}
\begin{proof}
Since $\PP(\Uv\in B(\x,\varepsilon))>0$ is a constant, it holds that
\begin{align*}\log(\PP(R>k|\Uv\in  B(\x,\varepsilon)))&=\log(\PP(R>k,\Uv\in  B(\x,\varepsilon)))-\log(\PP(\Uv\in  B(\x,\varepsilon)))\\
&\sim \log(\PP(R>k,\Uv\in  B(\x,\varepsilon))),
\end{align*}
as $k\to \infty$, and the claim of \eqref{eq_sconnectionlemma_eq1} is proved. The other claim follows if we multiply and divide the probability in the nominator of the latter presentation of \eqref{eq_sconnectionlemma_eq1} by $\PP(R>k)$.
\end{proof}
It is seen that the quantities in \eqref{eq_sconnectionlemma_eq1} are equal. However, the interpretations of the two forms are a bit different. The form on the left compares the decay rates of two tail functions. The form on the right does not have a tail function in its nominator, but this form makes monotony arguments and situations where $\PP(\Uv\in B(\x,\varepsilon))=0$ easier to handle.

If the quantities in \eqref{eq_sconnectionlemma_eq1} equal $1$, as they do in the definition of the set $S$, the quantity in \eqref{eq_sconnectionlemma_eq2} equals $0$. So, in principle, there are two ways in which a point can belong to the set $S$. A point belongs to $S$ if $B(\x,\varepsilon)$ has a positive probability under the limiting measure of Assumption \ref{as2} or if the limiting probability is $0$ but the function $-\log(\PP(\Uv\in  B(\x,\varepsilon)|R>k))$ grows to infinity slowly enough. Since we assume in \ref{as2} that the limit distribution is obtained after some $k_0$, the latter possibility is excluded by the assumptions.

Next, we prove general properties for the function $g$ defined in \eqref{eq_gdef}. To simplify notation, we denote in short
\begin{equation}\label{eq_isogdef}
G(A)=\lim_{k\to \infty} g(k,A) = \lim_{k\to \infty} \frac{\log(\PP(R>k,\Uv\in A))}{\log(\PP(R>k))}
\end{equation}
where $A\subset \Ss^{d-1}$ is a Borel set.

\begin{lem}\label{lem_g}
Suppose \ref{as1}-\ref{as3} hold. Then, the following properties hold for the function $g$ defined in \eqref{eq_gdef} and for the function $G$ defined in \eqref{eq_isogdef}. In the statements below, we assume that $k>k_R$ and $A\subset \Ss^{d-1}$ is a  Borel set.
\begin{enumerate}[label=(\roman*)]
    \item \label{point1_lemg} $g(k,\Ss^{d-1}) = 1$, $G(\Ss^{d-1})= 1$, $g(k,\emptyset) = \infty$ and $G(\emptyset)= \infty$.
    \item \label{point2_lemg} $g(k,A) \geq 1$ and $G(A)\geq 1$.
    \item \label{point3_lemg} Function $g$ is monotone in the sense that if $A\subset D\subset \Ss^{d-1}$, where $D$ is a Borel set, then $g(k,D)\leq g(k,A)$. In addition, $G(D)\leq G(A)$.
    \item \label{point4_lemg} Suppose $A_1,\dots, A_n\subset \Ss^{d-1}$ are Borel sets. Then, $G(A_1\cup\ldots\cup A_n)=\min(G(A_1),\ldots, G(A_n))$. In particular, $\min(G(A),G(A^c))=1$. 
\end{enumerate}
\end{lem}

\begin{proof}
\begin{itemize}
   \item[\ref{point1_lemg}] Since $\PP(R>k,\Uv\in\emptyset)=0$, it follows that $g(k,\emptyset)=\infty$ and $G(\emptyset)=\infty.$ The statements for $\Ss^{d-1}$ follow from the definition of $g$.
   \item[\ref{point3_lemg}] $A\subset D$ implies $\log(\PP(R>k,\Uv\in A))\leq \log(\PP(R>k,\Uv\in D))$. Dividing by the negative term $\log(\PP(R>k))$ yields the claims. 
   \item[\ref{point2_lemg}] Since $A\subset \Ss^{d-1}$, the statements follow from \ref{point3_lemg} and \ref{point1_lemg}.
   \item[\ref{point4_lemg}] Lemma 1.2.15 in \cite{Dembo1} combined with Assumption \ref{as3} yields
    \begin{eqnarray*}
    &&-G(A_1\cup\ldots\cup A_n)\\&\leq&\limsup_{k\to\infty} \frac{\log(\PP(R>k,\Uv\in A_1)+\ldots +\PP(R>k,\Uv\in A_n))}{-\log(\PP(R>k))} \\
    &=& \max\left(\limsup_{k\to\infty} \frac{\log(\PP(R>k,\Uv\in A_1))}{-\log(\PP(R>k))},\dots,
    \limsup_{k\to\infty} \frac{\log(\PP(R>k,\Uv\in A_n))}{-\log(\PP(R>k))}\right)\\
    &=& -\min\left(\lim_{k\to\infty}g(k,A_1),\dots,\lim_{k\to\infty}g(k,A_n)\right)\\
    &=& -\min (G(A_1),\ldots,G(A_n)).
    \end{eqnarray*}
    On the other hand, $G(A_1\cup\ldots\cup A_n)\leq G(A_i)$ for all $i=1,\ldots,n$ due to \ref{point3_lemg}. Thus, $G(A_1\cup\ldots\cup A_n)= \min( G(A_1),\ldots,G(A_n))$. The fact that $\min(G(A),G(A^c))=1$ follows from \ref{point1_lemg} because the sets $A$ and $A^c$ partition the set $\Ss^{d-1}.$
\end{itemize}
\end{proof}

In practical applications, the aim is to estimate $S$ from data. The set $S$ is a subset of the support of $\Uv$ in $\Ss^{d-1}$. The following results show that the set $S$ is not empty. 

\begin{lem}\label{lem_notempty}
Suppose \ref{as1}-\ref{as3} hold and $G(D)=1$ for a closed set $D\subset \Ss^{d-1}$. Then, $D\cap S\neq \emptyset$.
\end{lem}
\begin{proof}
Suppose $\mathcal{P}_1,\mathcal{P}_2,\ldots$ is a sequence of finite partitions of $D$. Assume further that all sets of the partitions are Borel sets and, for $n\geq 1$, $\mathcal{P}_{n+1}$ is a refinement of $\mathcal{P}_{n}$ such that the maximal diameter of the sets in $\mathcal{P}_{n}$ converges to $0$, as $n\to \infty$. For example, we could use dyadic partitions intersected with $D$.

Suppose $n\geq 1$ is fixed and consider the partition $\mathcal{P}_n=\{P_{n,1},P_{n,2},\ldots, P_{n,m_n}\}$ where the number of sets in $\mathcal{P}_n$ is denoted by $m_n$. Based on Part \ref{point4_lemg} of Lemma \ref{lem_g}, we know that 
$\min(G(P_{n,1}),\ldots, G(P_{n,m_n}))=G(D)=1$. That is, there is $i_n \in \{1,\ldots,m_n\}$ such that $G(P_{n,i_n})=1$.

The partition $\mathcal{P}_{n+1}$ is assumed to be a refinement of $\mathcal{P}_{n}$. So, the set $P_{n,i_n}$ is possibly partitioned into smaller sets and there is a subset, say $P_{n+1,i_{n+1}}\subset P_{n,i_n}$, where $i_{n+1} \in \{1,\ldots,m_{n+1}\}$ which satisfies $G(P_{n+1,i_{n+1}})=1$. Recall that the maximal diameters of the partitioning sets are assumed to converge to $0$. We see that there is a sequence of sets $P_{1,i_1}\supset P_{2,i_2}\supset \ldots$, where $P_{j,i_j}\in \mathcal{P}_j$ and $G(P_{j,i_j})=1$, for all $j=1,2,\ldots.$ Because $\Ss^{d-1}$ equipped with the geodesic metric is a complete metric space, there must be a limit point in the sequence of the sets. Let us denote the limit point by $\x$. Because the set $D$ is closed, the limit point $\x\in D$.

Suppose $\varepsilon>0$ is fixed. Suppose $n$ is so large that the maximal diameter of the sets in Partition $\mathcal{P}_n$ is less than $\varepsilon$. Then, by construction, there exists a set $P_{n,i_n}\in \mathcal{P}_n$ such that $P_{n,i_n}\subset B(\x,\varepsilon)$. Then, by monotony property \ref{point3_lemg} of Lemma \ref{lem_g}, we get that 
$$G(B(\x,\varepsilon))\leq G(P_{n,i_n})=1 $$
and the claim is proved because $\x$ belongs to the set $S$ by the definition of $S$.
\end{proof}

\begin{seur}\label{seur_S}
The set $S$ defined in \eqref{eq_domset} is not empty.
\end{seur}
\begin{proof}
Taking $D=\Ss^{d-1}$, it holds that $G(D)=1$ by Part \ref{point1_lemg} in Lemma \ref{lem_g} and thus the set $S$ is not empty by Lemma \ref{lem_notempty}.
\end{proof}
We make an assumption on the form of the set $S$ to rule out technically challenging cases that have little impact on practical applications.

\begin{enumerate}[label=(A4)]
    \item \label{as4} Firstly, we assume that there exists an open set in $S^c$. Secondly, we assume that $S$ in \eqref{eq_domset} can be written as 
    $$S=\overline{T_1}\cup T_2,$$ where $T_1$ is an open subset (possibly empty) of $\Ss^{d-1}$ and $T_2$ is a finite collection of individual points (possibly empty) of $\Ss^{d-1}$. We assume that each point in $T_2$ contains positive probability mass of the limit distribution of $\Uv\, |\, R>k$, as $k\to\infty$.
\end{enumerate}

\begin{rem} \label{rem_as4}
Assumption \ref{as4} implies that not all directions have the same riskiness. The assumption also ensures that $S$ does not contain continuous subsets in lower dimensions than $d-1$. It admits directly e.g.\ distributions where the riskiest direction is concentrated to a cone or a single vector. Even if the original distribution of $\X$ does not satisfy Assumption \ref{as4}, it is possible to construct a new approximating distribution by adding a small independent continuous perturbation to the vector $\Uv$ to obtain a new distribution that satisfies \ref{as4}. The perturbation could be, for example, a random variable that has the uniform distribution on a small ball.
\end{rem}

If there exists a joint density of $(R,\Uv)$ or if $\Uv$ is discrete, we can write the conditional risk or hazard function as
\begin{displaymath}
h(k,\x)=-\log(\PP(R>k \ |\ \Uv=\x)),
\end{displaymath}
where $h(k,\x)$ is a positive, increasing function for fixed $\x\in \Ss^{d-1}$. The notation admits presenting the conditional risk function in the following simplified way in typical cases. 
\begin{exa} \label{exa_conditional_risk}
\begin{enumerate}
    \item If the random vector $\X$ is elliptically distributed with $R=\|\X\|_2$ and $\Uv=\frac{\X}{\|\X\|_2}$, its conditional risk function can be written as 
    \begin{displaymath}
    \log(\PP(R>k \ |\ \Uv=\x))= c(\x) h(k).
    \end{displaymath}
    Here, $c \colon \Ss^{d-1} \to \RR$  is a function such that $c(\vv)=c(-\vv)$ holds for all $\vv\in \Ss^{d-1}$. The function $c$ is a continuous map of $\Ss^{d-1}$ to an interval. In general, we set $h(k)$ such that $\min c(\x) =1$, so $h(k)$ is the risk function in the riskiest direction. 
    In the special case, where the distribution of $\X$ is spherical $c(\x)$ is constant. 
    \item There can be different tail behaviour in different subsets of $\Ss^{d-1}$. If there exists a finite partition $A_1,\dots,A_m$ of $\Ss^{d-1}$ such that the risk function does change given $\x \in A_j$, we can write
    \begin{displaymath}
    h(k,\x) = \sum_{j=1}^m h_j(k)\ind(\x\in A_j),
    \end{displaymath}
    where $h_j(k)$ refers to the risk function in the direction of the set $A_j$.
\end{enumerate}
\end{exa}

\subsection{Subsets of the unit sphere}\label{sec_geodesic}
For $\x,\y\in\Ss^{d-1},$ we define the geodesic metric $\dist(\x,\y)$ by
$$\dist(\x,\y)=\{\textrm{length of the shortest geodesic connecting } \x \textrm{ and } \y\}.$$
In this metric, open balls are subsets of $\Ss^{d-1}$ denoted by $B(\x,r)$, where the vector $\x\in \Ss^{d-1}$ is the centre of the open ball and $r>0$. So $B(\x,r)=\{\y\in \Ss^{d-1}: \dist(\x,\y)<r\}$. The corresponding closed ball is denoted by $\overline{B}(\x,r)$. 
Note that the shape of the ball $B(\x,r)$ depends on the used norm.

\begin{defn}\label{def_delta-swelling}
For any set $A\subset \Ss^{d-1}$, we call the set 
\begin{displaymath}
A^\delta =\left\{\x\in \Ss^{d-1}: \dist(\x,A)<\delta\right\},
\end{displaymath}
the geodesic $\delta$-swelling of the set $A$. Here, $\dist(\x,A)$ is the geodesic distance of $\x$ to the set $A$, $\dist(\x,A)=\inf_{\av\in A}\dist(\x,\av)$.
\end{defn}

By
\begin{displaymath}
\dist_H(A,B)= \max\left\{\sup_{\av\in A}\dist(\av,B),\sup_{\bv\in B}\dist(\bv,A)\right\}
\end{displaymath}
we denote the Hausdorff distance, where $\dist(\av,B)$ is defined as in Definition \ref{def_delta-swelling}.

\section{Minimal set of riskiest directions} \label{sec_theory}
Our aim is to find the riskiest directions. We search for the minimal set that dominates the tail behaviour of the random vector $\X$ in the sense of \eqref{eq_domset}. To this end, we need to identify sets $A\subset \Ss^{d-1}$ for which, given $\delta>0$, the inequality
\begin{displaymath}
\log(\PP(R>k,\Uv\in A))>\log(\PP(R>k, \Uv \in (A^\delta)^c))
\end{displaymath}
holds for all $k$ large enough. The inequality demands a positive probability measure of $A$.

For the next result, we define the collection $\mathcal{A}$ of testing sets $A\subset  \Ss^{d-1}$ as follows. A set $A$ is an element of $ \mathcal{A}$ if $A$ is a finite union of open balls such that for all $\x\in A^c$ and for all $\varepsilon >0$ the open ball $B(\x,\varepsilon)$ contains an open ball $B$ that belongs to $A^c$. Note that the point $\x$ does not have to be in the set $B$. In particular, this guarantees that $A^c$ does not contain any isolated points. 

\begin{thm}\label{thm_minimalset}
Let $\X=R\Uv\in \RR^d, d\geq 2$, be such that Assumptions \ref{as1}-\ref{as4} hold. Set
\begin{equation}\label{eq_tildes}
\tilde{S} = \cap \left\{A\in \mathcal{A}:  \lim_{k\to \infty} g(k,A) <  \lim_{k\to \infty} g(k,A^c)\right\}.
\end{equation}
Then, $S=\textrm{cl}(\tilde{S})$, where $S$ is as in \eqref{eq_domset}. Furthermore, for all $\delta>0$,
\begin{equation}\label{claim1}
\lim_{k\to\infty}\frac{\log\left(\PP(R>k,\Uv\in \tilde{S}^\delta)\right)}{\log(\PP(R>k))}=1
\end{equation}
and 
\begin{equation}\label{claim2}
\lim_{k\to\infty}\frac{\log\left(\PP(R>k,\Uv\in (\tilde{S}^\delta)^c)\right)}{\log(\PP(R>k))}>1.
\end{equation}
\end{thm}

\begin{proof}
The proof of the theorem is performed in steps.
\begin{enumerate}
    \item\label{part1} \textsc{We claim:} $S\subset \textrm{cl}(\tilde{S})$.
    
    Recall that under Assumption \ref{as4}, the set $S$ can be written as the union of $\overline{\textrm{int}(S)}$ and a finite number of individual points. We have two cases to cover.
    
    First, we study the case where the set $T_2$ of Assumption \ref{as4} contains a point. Suppose there is an individual point $\x\in \Ss^{d-1}$ such that
    \begin{equation}\label{eq_proofeq1}
    1=\lim_{k\to \infty} \frac{\log(\PP(R>k|\Uv= \x))}{\log(\PP(R>k))}=\lim_{k\to\infty}g(k,\{\x\}),
    \end{equation}
    where, the latter equation follows from Lemma \ref{lem_S}.
    Then, $\x$ cannot belong to $A^c$ for any set $A\in \mathcal{A}$ that satisfies the inequality in \eqref{eq_tildes}. To see this, assume in the contrary that $\x$ belongs to $A^c$ for some set $A\in \mathcal{A}$. Then, due to monotony mentioned in \ref{point3_lemg} in Lemma \ref{lem_g},
    $$\lim_{k\to\infty}g(k,A^c)\leq\lim_{k\to\infty}g(k,\{\x\})=1$$
    by Equality \eqref{eq_proofeq1}. Since $\lim_{k\to\infty}g(k,A)$ cannot be strictly less than $1$ by Part \ref{point2_lemg} in Lemma \ref{lem_g}, the inequality cannot be true if $\x$ is in the complement of a testing set $A\in \mathcal{A}$. We conclude that $\x$ must belong to all testing sets $A\in \mathcal{A}$ that satisfy the inequality in \eqref{eq_tildes}. So, $\x \in \tilde{S}$.
    
    Next, we consider the case where the set $T_1$ of Assumption \ref{as4} contains a point. Suppose $\x\in \textrm{int}(S)$. Then, there is a number $\varepsilon>0$ such that $B(\x,\varepsilon)\subset \textrm{int}(S)$. Let $A\in \mathcal{A}$. We show that $\x$ cannot be in $A^c$ for a set $A$ that satisfies the inequality in \eqref{eq_tildes}. Assume in the contrary that $\x\in A^c$. In this situation, we can find a small ball which is entirely in the intersection $B(\x,\varepsilon)\cap A^c$. To see this, let $\varepsilon'=\varepsilon/2$. By the definition of the testing sets in $\mathcal{A}$, the ball $B(\x,\varepsilon')$ contains another open ball, say $B(\x',\varepsilon'')$ which is contained in $A^c$. Since $B(\x,\varepsilon')\subset B(\x,\varepsilon)$, we see that $B(\x',\varepsilon'')\subset B(\x,\varepsilon)\cap A^c\subset \textrm{int}(S)\cap A^c$. So, because $\x'\in S$, the limit in \eqref{eq_domset} applied for radius $\varepsilon''$ states that
    $$\lim_{k\to \infty} \frac{\log(\PP(R>k|\Uv\in B(\x',\varepsilon'')))}{\log(\PP(R>k))}=\lim_{k\to\infty}                 g(k,B(\x',\varepsilon''))=1. $$
    So, due to monotony mentioned in \ref{point3_lemg} in Lemma \ref{lem_g},
    $$\lim_{k\to\infty}g(k,A^c)\leq\lim_{k\to\infty}g(k,B(\x',\varepsilon''))=1.$$
    In conclusion, any testing set $A\in \mathcal{A}$ that does not contain $\x$ cannot satisfy the inequality in \eqref{eq_tildes}. So, all testing sets that satisfy the inequality must contain $\x$. So, $\x \in \tilde{S}$.
    
    The above deductions imply, using the notation of Assumption \ref{as4}, that $\textrm{int}(S)\cup T_2 \subset \tilde{S}$ which implies $\textrm{cl}(\textrm{int}(S)\cup T_2)=S\subset \textrm{cl}(\tilde{S})$ and the claim is proved.
    
    \item \label{part2}\textsc{We claim:} $\textrm{cl}(\tilde{S})\subset S$. 
    
    The claim is equivalent with $S^c\subset \textrm{cl}(\tilde{S})^c$.
    Let $\x\in S^c$. Then, by the definition of $S$ in \eqref{eq_domset}, either there exists $\varepsilon>0$ such that $\PP(\Uv\in B(\x,\varepsilon))=0$ or there exists $\varepsilon>0$ such that
    \begin{eqnarray}\label{eq_proof}
    1&<&\lim_{k\to\infty}g(k, B(\x,\varepsilon)).
    \end{eqnarray}
    In the latter case, Inequality \eqref{eq_proof} also holds when $\varepsilon$ is replaced by any $\varepsilon'<\varepsilon$ due to the monotony, see \ref{point3_lemg} in Lemma \ref{lem_g}. 
    
    In order to show that $\x\in \tilde{S}^c$, it suffices to find one testing set $A\in \mathcal{A}$ such that $\x\in A^c$ and $A$ satisfies the inequality of \eqref{eq_tildes}. This is because the set $\tilde{S}$ is the intersection of all testing sets that satisfy the inequality. 
    
    Let the number $\varepsilon$ be such that $\PP(\Uv\in B(\x,\varepsilon))=0$ or \eqref{eq_proof} holds. Now, setting formally $A'=\Ss^{d-1}\backslash B(\x,\varepsilon)$ fulfils the inequality of \eqref{eq_tildes} but this $A'$ is not a member of the collection $\mathcal{A}$. However, we can construct a set $A\in \mathcal{A}$ using a finite number of open balls such that $A$ covers the set $\Ss^{d-1}\backslash B(\x,\varepsilon)$ but does not intersect the set $B(\x,\varepsilon/2)$. 
    
    With this set $A$, we have that 
    \begin{equation}\label{eq_proofeq2}
        \lim_{k\to \infty} g(k,A) <  \lim_{k\to \infty} g(k,A^c)
    \end{equation} because $A^c\subset B(\x,\varepsilon)$. More precisely, by monotony, 
    $$\lim_{k\to \infty} g(k,A^c) \geq   \lim_{k\to \infty} g(k,B(\x,\varepsilon))>1.$$ 
    The limit $\lim_{k\to \infty} g(k,A)$ equals $1$ by \ref{point4_lemg} in Lemma \ref{lem_g} and so Inequality \eqref{eq_proofeq2} holds. In conclusion, $\x$ belongs to the complement of this $A$ and consequently $\x\in \tilde{S}^c$. 
    We have shown $S^c \subset \tilde{S}^c$ which is equivalent with $\tilde{S}\subset S$ which implies $\textrm{cl}(\tilde{S})\subset \textrm{cl}(S)=S.$
    
    \item \textsc{We claim that \eqref{claim1} and \eqref{claim2} hold}.
    
    Let $\delta>0$ be fixed. By Parts \ref{part1}-\ref{part2} of the proof, we know that 
    \begin{equation}\label{eq_proofeq4}
        S\subset \tilde{S}^\delta.
    \end{equation}Assume in the contrary to Claim \eqref{claim2} that 
    \begin{equation}\label{eq_proofeq3}
    \lim_{k\to\infty}\frac{\log\left(\PP(R>k,\Uv\in (\tilde{S}^\delta)^c)\right)}{\log(\PP(R>k))}=1.
    \end{equation}
    In particular, the set $(\tilde{S}^\delta)^c$ is closed, so it must contain a point of $S$ by Corollary \ref{seur_S}. This impossible by \eqref{eq_proofeq4} and thus \eqref{eq_proofeq3} cannot hold as an equality. So, Inequality \eqref{claim2} and consequently, by \ref{point4_lemg} in Lemma \ref{lem_g}, Equality \eqref{claim1} hold.
    \qedhere
\end{enumerate}
\end{proof}

In the case of elliptically distributed random vectors, the minimal set that dominates the tail behaviour of the random vector might consist only of singletons that do not have probability mass. One can still approximate the distribution using the method of Remark \ref{rem_as4}. 
In this example, we use the $l_2$-norm.
\begin{exa}\label{exa_elliptic}
Let the random vector $\X$ be elliptically distributed such that $h(k,\uv)=c(\uv)h(k)$ where $c$ is a continuous function on $\Ss^{d-1}$ that achieves its minimum only at $\x$ and $-\x\in \Ss^{d-1}$ and assume that $\x$ and $-\x$ are points in the set $T_2$ of Assumption \ref{as4}. Then, for all $\varepsilon>0$, it holds that
\begin{displaymath}
\lim_{k\to\infty} g(k,(B(\x,\varepsilon)\cup B(-\x,\varepsilon))\cap\Ss^{d-1})<\lim_{k\to\infty} g(k,\Ss^{d-1}\backslash(B(\x,\varepsilon)\cup B(-\x,\varepsilon))).
\end{displaymath}
So, $S=\{\x,-\x\}$. Choosing the risk function of $R$ to be $h(k)$, the minimum $\min c(\vv)$ equals one.
\end{exa}

\begin{seur}\label{seur_AsupsetS}
Let $\X=R\Uv$ be such that Assumptions \ref{as1}-\ref{as4} hold and $A\subset \Ss^{d-1}$ is a Borel set. If there exists $\delta>0$ such that $\tilde{S}^\delta\subset A$, it holds that
\begin{equation} \label{eq_inequalityA}
    \lim_{k\to\infty} g(k,A)<\lim_{k\to\infty}g(k,A^c).
\end{equation}
\end{seur}
\begin{proof}
The claim follows from \eqref{claim1} and \eqref{claim2} together with the monotony of the function $g$.
\end{proof}

\section{Towards estimators}\label{sec_estimator}
In this section, we introduce procedures that can be used to form estimators for the set $S$ based on data. We do not present such estimators explicitly here, but the results can be used as a theoretical basis for this work. Throughout the section, we use the $l_2$-norm, so the geodesic distance or great ball distance on the unit sphere $\Ss^{d-1}$ is defined as $\textrm{dist}(\x,\y)=\arccos(\x\cdot\y)$, where $\x\cdot\y$ is the dot product of $\x$ and $\y$. The metric $\text{dist}$ is discussed in detail for instance in Proposition 2.1 of \cite{Bridson}.

The following lemmas are auxiliary results for Theorem \ref{thm_estimator}.
\begin{lem}\label{lem_partition_S}
Let $\Ss^{d-1}$ be the unit sphere of $\RR^d$ equipped with the $l_2$-norm and for $\x,\y\in \Ss^{d-1}$ let $\textrm{dist}(\x,\y)=\arccos(\x\cdot\y)$ be the geodesic distance or the great ball distance on $\Ss^{d-1}$. We define $\overline{B}(\x,0)=\{\x\}$ for $\x \in \Ss^{d-1}$.

Then, for fixed $\x\in\Ss^{d-1}$ and $0<r\leq \pi$, the balls $B(\x,r)$ and $\overline{B}(-\x,\pi-r)$ partition the unit sphere.
\end{lem}
\begin{proof}
Since $\x\in \Ss^{d-1}$, it holds $1=\|\x\|_2=\|-\x\|_2$ so $-\x\in \Ss^{d-1}$.
To prove the claim, we show that $\overline{B}(-\x,\pi-r)$ is the complement of $B(\x,r)$ in $\Ss^{d-1}$. Since $B(\x,r)=\{\y\in \Ss^{d-1}: \textrm{dist}(\x,\y)<r\}$ its complement is the set $\{\y\in \Ss^{d-1}:\textrm{dist}(\x,\y)\geq r\}$. The condition $\textrm{dist}(\x,\y)\geq r$ is by definition
\begin{equation}\label{cond_arccos}
    \arccos(\x\cdot\y)\geq r.
\end{equation}
Due to the fact that for all $x\in [-1,1]$, $\arccos(x)+\arccos(-x)=\pi$, the condition \eqref{cond_arccos} is equivalent with $\arccos(-\x\cdot\y) \leq \pi-r$ because $\x$ and $\y$ are unit vectors. The last expression can be written as $\textrm{dist}(-\x,\y)\leq \pi-r$. So $\overline{B}(-\x,\pi-r)$ is the complement of $B(\x,r)$ in $\Ss^{d-1}$.
\end{proof}

\begin{lem}\label{lem_apu}
Let $\Ss^{d-1}$ be the unit sphere of $\RR^d$ equipped with the $l_2$-norm and for $\x,\y\in \Ss^{d-1}$ let $\dist(\x,\y)=\arccos(\x\cdot\y)$ be the geodesic distance.

Then, for $\x,\y\in \Ss^{d-1}, \x\neq \y, \x\neq -\y$ there exists $\delta>0$ so that the intersection $B(\y,\dist(\x,\y)+\delta/2)\cap B(\x,\delta)$ contains an open set of $\Ss^{d-1}$ and $\overline{B}(\y,\dist(\x,\y)+\delta/2)^c\cap B(\x,\delta)$ contains an open set of $\Ss^{d-1}$.
\end{lem}
\begin{proof}
Since $\y\neq -\x$ and $\y\neq \x$ it holds that $0<\dist(\x,\y)<\pi$ due to the definition of the great ball distance. Take $\delta<\dist(-\x,\y)$. By the choice of $\delta$, $B(\y,\dist(\x,\y)+\delta/2)$, $B(\x,\delta)$ and $\overline{B}(\y,\dist(\x,\y)+\delta/2)^c$ are proper subsets of the unit sphere.
Since all three sets are open, it remains to show that there is a point in each of the intersections of the claim. We recover the points explicitly.

Since $\arccos(x)+\arccos(-x)=\pi$, it holds $\pi=\dist(\x,-\x)=\dist(\x,\vv)+\dist(-\x,\vv)$ for any $\vv\in \Ss^{d-1}.$ For $\y\neq \x, \y\neq -\x$ there exists a unique minimising geodesic between $\x$ and $\y$. 
For all points $\z$ that lie on this minimising geodesic between $\x$ and $-\y$ we have that $\dist(\x,-\y)=\dist(\x,\z)+\dist(\z,-\y)$.

Taking $\z$ on this geodesic such that $\dist(\z,\x)=\frac{1}{4}\delta$, it holds by definition that $\z\in B(\x,\delta)$. Because
\begin{eqnarray*}
\dist(\y,\z)&=&\pi-\dist(\z,-\y)=\pi-(\dist(\x,-\y)-\dist(\x,\z))\\
&=&\dist(\x,\y)+\frac{1}{4}\delta< \dist(\x,\y)+\delta/2,
\end{eqnarray*} 
we also have $\z  \in B(\y,\dist(\x,\y)+\delta/2)$. In conclusion, $\z \in B(\y,\dist(\x,\y)+\delta/2)\cap B(\x,\delta)$.

On the other hand, taking $\z'$ on the same minimising geodesic between $\x$ and $-\y$ such that $\dist(\z',\x)=\frac{3}{4}\delta$, it holds by definition that $\z'\in B(\x,\delta)$. By similar calculations as above, it holds that 
$$\dist(\y,\z')=\dist(\x,\y)+\dist(\x,\z')> \dist(\x,\y)+\delta/2,$$
so  $\z'\in \overline{B}(\y,\dist(\x,\y)+\delta/2)^c$ and hence $\z'\in \overline{B}(\y,\dist(\x,\y)+\delta/2)^c\cap B(\x,\delta)$.
Thus, both intersections contain a point and therefore an open set.
\end{proof}

\subsection{Algorithm}\label{sec_prelim}
In this section, we present an algorithm to find the minimal set $S$ that dominates the tail behaviour of the studied random vectors under suitable assumptions. To this end, we study the function $G$ defined in \eqref{eq_isogdef} for different sets. Due to Theorem \ref{thm_minimalset}, the minimal set $S$ in \eqref{eq_domset} that dominates the tail behaviour of the random vector is contained in the intersection of all testing sets that fulfil the inequality in Condition \ref{eq_tildes}. We present a theoretical procedure for finding the set of the riskiest directions. 

\begin{algo}
We start by defining a mapping $\vv \mapsto A_\vv$ where $A_\vv$ is an open set and $\vv$ is an element on the unit sphere $\Ss^{d-1}$. The algorithm for finding the minimal set $S$ that dominates the tail behaviour of the random vectors is presented in two steps.
\begin{enumerate}\label{algo}
    \item \label{step3} Let $\vv\in \Ss^{d-1}$. 
    
    If, for some $r<\pi$, it holds that $G(B(\vv,r)^c)>1$ and $G(B(\vv,r))=1$ then define $A_\vv= B(\vv,r_\vv)$, where   
        $r_{\vv}$ is the smallest radius fulfilling the condition
        \begin{equation}\label{eq_algo_1}
        \lim_{k\to\infty} g(k,B(\vv,r_\vv)^c)>1.
        \end{equation} 
        In other words,
        \begin{displaymath}
        r_\vv = \inf \{r>0: \lim_{k\to\infty} g(k,B(\vv,r)^c)>1\}
        \end{displaymath}
        and $B(\vv,r_\vv)$ is the smallest ball centered around $\vv$ that contains $S$.
        
    If $G(B(\vv,r)^c)=1$, for all $0<r<\pi$, set $A_\vv=\Ss^{d-1}$ and if $G(B(\vv,r)^c)>1$ for all $0<r<\pi$, set $A_\vv=\emptyset$. 
    \item Set
    \begin{eqnarray*}
    \hat{S} = \bigcap_{\vv\in \Ss^{d-1}} A_\vv.
    \end{eqnarray*}
\end{enumerate}
\end{algo}
The sets $A_\vv=B(\vv,r_\vv)$ are open balls and belong by definition to the set of testing sets $\mathcal{A}$. Furthermore, the entire unit sphere and the empty set are open sets as well and thus belong to $\mathcal{A}$. 
Due to Lemma \ref{lem_partition_S}, $B(\vv,r_\vv)^c=\overline{B}(-\vv,\pi-r_\vv)$, Equation \eqref{eq_algo_1} can be rewritten in equivalent form
$$\lim_{k\to\infty} g(k,\overline{B}(-\vv,\pi-r_\vv))>1.$$
So, $r_\vv$ is the smallest radius such that $G(\overline{B}(-\vv,\pi-r_\vv))>1$.

The following lemma shows the connection between the choice of the set $A_\vv$ and the vector $-\vv$ which points to the opposite direction of $\vv$.
\begin{lem} \label{lem_A_v}
Suppose Assumptions \ref{as1}-\ref{as4} hold. Then, it holds for any $\vv\in \Ss^{d-1}$ and its corresponding set $A_\vv$ defined in Algorithm \ref{algo} that $A_\vv = \Ss^{d-1}$ if and only if $-\vv\in S$. 
\end{lem}
\begin{proof}
To prove $A_\vv=\Ss^{d-1}$ is equivalent to $-\vv\in S$, we show that $-\vv\in S$ implies $A_\vv=\Ss^{d-1}$ and $-\vv\notin S$ results in $A_\vv \subsetneq \Ss^{d-1}$.

Let $-\vv\in S$. Then, by the definition of $S$ for all $\varepsilon>0$ it holds that  $\PP(\Uv\in B(-\vv,\varepsilon))>0$ and 
\begin{displaymath}
\lim_{k\to\infty} \frac{\log(\PP(R>k|\Uv\in B(-\vv,\varepsilon)))}{\log(\PP(R>k))}=1
\end{displaymath}
which is equivalent to $G(B(-\vv,\varepsilon))=1$ due to Lemma \ref{lem_S}. As a consequence, the algorithm chooses $A_\vv=\Ss^{d-1}$ because there does not exist $\varepsilon>0$ such that $G(B(-\vv,\varepsilon))>1$ and $B(-\vv,\varepsilon)=\overline{B}(\vv,\pi-\varepsilon)^c$ by Lemma \ref{lem_partition_S}.

On the other hand, let $-\vv\notin S$. Then, either $\PP(\Uv\in B(-\vv,\varepsilon))=0$ for some $\varepsilon>0$ or for all $\varepsilon>0$ it holds that $\PP(\Uv\in B(-\vv,\varepsilon))>0$ and there exists $\varepsilon'>0$ such that
\begin{equation}\label{ineq1}
\lim_{k\to\infty} \frac{\log(\PP(R>k|\Uv\in B(-\vv,\varepsilon')))}{\log(\PP(R>k))}>1.
\end{equation}
In the first case, $\PP(\Uv\in B(-\vv,\varepsilon))=0$ implies $G(B(-\vv,\varepsilon))>1$ and in the second case, Inequality \eqref{ineq1} is equivalent to $G(B(-\vv,\varepsilon')) >1$ by Lemma \ref{lem_S}. 
Due to Lemma \ref{lem_partition_S}, $B(-\vv,\varepsilon')=\overline{B}(\vv,\pi-\varepsilon')^c$ so $G(\overline{B}(\vv,\pi-\varepsilon')^c)>1$ and by Lemma \ref{lem_g} it holds that $G(\overline{B}(\vv,\pi-\varepsilon'))=1$. Hence, $r_\vv\leq \pi-\varepsilon'$ and $A_\vv=B(\vv,r_\vv)\subsetneq \Ss^{d-1}$ or $A_\vv=\emptyset$.
\end{proof}

The estimator does not detect all possible sets. Recall, that in general we assume in \ref{as4} $S=\overline{T_1}\cup T_2$, where $T_1$ is an open subset of $\Ss^{d-1}$ and $T_2$ is a finite collection of individual points. For example, if $S=T_2$ or $S=\overline{T_1}\cup T_2$ and $T_2$ is not empty, Algorithm \ref{algo} will not detect the set $T_2$. 
If $S$ contains only a singleton, so $S=\{\vv\}$ for some $\vv\in \Ss^{d-1}$ it holds that $G(B(\vv,\varepsilon))=1$ and $G(B(\vv,\varepsilon)^c)>1$ for all $\varepsilon>0$. Then, the algorithm sets $A_\vv=\emptyset$ so the estimator is empty. In general, it can be seen that the algorithm does not detect any finite number of individual points in $S$. However, the procedure described in Remark \ref{rem_as4} can be used to modify data sets in order to avoid problems in practice.

The estimator $\hat{S}$ has the capacity to detect sets $S$ that are not necessarily convex or even connected. The set $S$ can be, for example, a disjoint union of open sets. 
\begin{exa}
A classical football is made of 12 black pentagons and 20 white hexagons. Assume that the directions $\Uv$ of the random vectors are uniformly distributed on the surface of the football. Furthermore, assume that random variables $R$ connected with open black pentagons have a much heavier tail then random variables $R$ connected with closed white hexagons such that $G(\textrm{white part of football})>1$ and $G(\textrm{black part of the football})=1$. If we choose $\vv\in \Ss^2$ such that $-\vv$ points in the centre of a black pentagon, it holds $G(B(-\vv,\varepsilon))=1$ for any $\varepsilon>0$, so the algorithm sets $A_\vv=\Ss^{2}$. If we choose $\vv\in \Ss^2$ such that $-\vv$ points in the centre of a white hexagon, it holds for its closed inscribed circle $\overline{B}(-\vv,r)$ that $G(\overline{B}(-\vv,r))>1$ so $A_\vv$ does not contain this closed inscribed circle and thus also $\hat{S}$ does not contain it. If we choose $\vv\in \Ss^2$ such that $-\vv$ is the centre of an edge of two white hexagons, it holds that $G(\overline{B}(-\vv,r))>1$ where $r$ is half of the length of the edge. Therefore, the edge is not included in the set $A_\vv$ and thus the edge is not included in $\hat{S}$. All in all, the intersection of sets $A_\vv$ where $\vv$ are such that $-\vv$ points either in the direction of the centre of a white hexagon or in the direction of the centre of an edge between two white hexagons returns a set that is not connected. Taking the intersection over all $\vv\in\Ss^{2}$, Algorithm \ref{algo} would return the union of the closed black pentagons as $\hat{S}$.
\end{exa}

To avoid the problem with individual points, we make a simplifying assumption on $S$. 
\begin{thm}\label{thm_estimator}
Suppose Assumptions \ref{as1}-\ref{as3} and \ref{as4} with $S=\overline{T_1}$. In particular, the set $S$ does not contain any individual points.

Then, it holds that 
\begin{equation*}
    S=\textrm{cl}(\hat{S}),
\end{equation*}
where $\hat{S}$ is as in Algorithm \ref{algo}.
\end{thm}

\begin{proof}
The proof of the theorem is performed in steps.
\begin{enumerate}
    \item We claim: $S\subset \textrm{cl}(\hat{S}).$
    
We show $\x\in \textrm{int}(S)$ implies $\x\in \textrm{cl}(\hat{S})$ and take then the closure of the sets to prove the claim.

Note first that, by Assumption \ref{as4}, $S$ is a proper subset of $\Ss^{d-1}$. Let $\x$ be in the interior of $S$. Then, there exists some $\delta>0$ such that $B(\x,\delta)\subset S$. Furthermore, for all $\varepsilon>0$ it holds that
\begin{displaymath}
\lim_{k\to\infty} \frac{\log(\PP(R>k|\Uv\in B(\x,\varepsilon)))}{\log(\PP(R>k))}=1
\end{displaymath}
which is equivalent to $G(B(\x,\varepsilon))=1$ due to Lemma \ref{lem_S}.

We need to show that $\x\in A_\vv$ for all $\vv\in \Ss^{d-1}$.
In the algorithm, the set $A_\vv$ can be the empty set, the unit sphere or an open ball with centre $\vv$. If $A_\vv=\Ss^{d-1}$, it contains $\x$ by default. 

Let $\vv\in \Ss^{d-1}$ be fixed. We show $\x \in A_\vv$. There are different cases to consider.
\begin{enumerate}
\item If $\vv=\x$, the set $B(\x,\delta/2)^c$ contains an open subset of $S$ since 
\begin{displaymath}
\emptyset \neq B(\x,\delta/2)^c \cap B(\x,\delta) \subset B(\x,\delta) \subset S
\end{displaymath}
so $G(B(\x,\delta/2)^c)=1$ and thus $r_\x\geq \delta/2$. It follows that $A_\x$ cannot be empty and $\x\in A_\x$.

\item If $\vv=-\x$, it holds that $A_\vv=\Ss^{d-1}$ by Lemma \ref{lem_A_v} so $\x\in A_{-\x}$.

\item \label{part_c} If $\vv\in B(\x,\delta), \vv\neq \x$ there exists some $\delta'>0$ such that $B(\vv,\delta')\subset B(\x,\delta)$ so both sets $B(\vv,\delta')$ and its complement contain an open subset of $S$ and thus $r_\vv\geq \delta'$ so $A_\vv\neq \emptyset$.

It holds by Lemma \ref{lem_apu} that both sets $B(\x,\delta)\cap B(\vv,\textrm{dist}(\x,\vv)+\delta/2)$ and $B(\x,\delta)\cap B(\vv,\textrm{dist}(\x,\vv)+\delta/2)^c$ are not empty. By the monotony of $G$, it holds that $G(B(\x,\delta)\cap B(\vv,\textrm{dist}(\x,\vv)+\delta/2)^c)\leq G(B(\x,\delta))=1$ so $r_\vv \geq \textrm{dist}(\x,\vv)+\delta/2$ and $\x\in A_\vv$.

\item If $\vv\notin B(\x,\delta), \vv\neq -\x$, there exists by Lemma \ref{lem_apu} $\delta'<\min(\delta, \dist(-\x,\vv))$ such that the intersections $B(\vv,\textrm{dist}(\x,\vv)+\delta'/2)\cap B(\x,\delta)$ and $B(\vv,\textrm{dist}(\x,\vv)+\delta'/2)^c \cap B(\x,\delta)$ are not empty and in particular contain both an open subset. Since $B(\x,\delta)\subset S$, it holds that $A_\vv$ cannot be empty. With similar deduction as in Part \ref{part_c} it follows that $\x\in A_\vv$.
\end{enumerate}

\item We claim: $\textrm{cl}(\hat{S})\subset S.$

 Let $\x\notin S$. We show that there exists $\vv$ such that $\x \notin A_\vv$. More specifically, we can choose $\vv=-\x$.
 
 If $\PP(\Uv\in B(\x,\varepsilon))=0$ for some $\varepsilon>0$ it holds that $G(B(\x,\varepsilon))>1$. Additionally, if $\PP(\Uv\in B(\x,\varepsilon))>0$ for all $\varepsilon>0$ it holds by the definition of $S$ in \eqref{eq_domset} that there exists some $\varepsilon>0$ such that 
\begin{equation}\label{ineq_notins}
\lim_{k\to\infty} \frac{\log(\PP(R>k|\Uv\in B(\x,\varepsilon)))}{\log(\PP(R>k))}>1.
\end{equation}
Inequality \eqref{ineq_notins} is equivalent to $G(B(\x,\varepsilon))>1$, as well. By Lemma \ref{lem_g}, it holds in both cases that $G(\overline{B}(\x,\frac{\varepsilon}{2}))>1$. The set $A_{-\x}$ is of the form $B(-\x,r_{-\x})$ with $r_{-\x}\leq \pi-\varepsilon/2$ and it is in particular not the set $\Ss^{d-1}$. Thus $\x \notin A_\vv$ when $\vv=-\x$. 
\end{enumerate}
\end{proof}

\section{Applications and examples}\label{sec_applications}

In practical applications one has only a finite number of observations. Thus, an approximation based on the theoretical algorithm is required. Next, we present a starting point for the formulation of estimators and study how they perform with data. 

We define an empirical version of the function $g$ defined in \eqref{eq_gdef}. Given observations $\x_1,\x_2,\dots, \x_n$ in $\RR^d$, the empirical version of $g$ is denoted by 
\begin{equation}\label{eq_g_hattu}
    \hat{g}(k,A)= \frac{\log\left(\sharp\{i: \|\x_i\|>k, \frac{\x_i}{\|\x_i\|} \in  A\}/n\right)}{\log\left(\sharp\{i: \|\x_i\|>k\}/n\right)},
\end{equation}
where $k>0$ and $A$ is a Borel set on the unit sphere.

In general, if
\begin{equation} \label{eq_testineq}
    \hat{g}(k,A)>1+c
\end{equation}
holds for a set $A$ and some $c>0$, it gives evidence for $A$ being in the complement of $S$. 
If
\begin{equation} \label{eq_testineq2}
   \hat{g}(k,A)<1+c
\end{equation}
we gain evidence for $A$ containing at least some subset of $S$. Since we can calculate the values of $\hat{g}$ for any set, the challenge is to perform the calculations in a systematic way and combine the results to form an estimate for $S$. The most practical choices for the sets $A$ appear to be open balls centered around a given point on the unit sphere.

\subsection{Simulation study}\label{sec_simulation}

A simulation study with $n=800000$ observations was performed. Assumptions $\ref{as1}-\ref{as4}$ are valid by construction. A two-dimensional data set was produced where heavy-tailed observations are possible in all directions, but some directions are heavier than others. The space was split into $8$ equally sized cones. In each cone, the radii of the observations have the same distribution and the directional components are uniformly distributed within each cone. The original data set is presented in Figure \ref{pic_orig}. The red lines indicate how the space is split. Two of the sectors have heavier Pareto distributed radial components and the rest have lighter Weibull distributed components. In Figure \ref{pic_orig}, the cones with Pareto distributed radial components are the ones with the largest observations measured in the $l_2$-norm.

\begin{figure}[htb]
     \centering
     \begin{subfigure}[b]{0.49\textwidth}
         \centering
         \includegraphics[width=\textwidth]{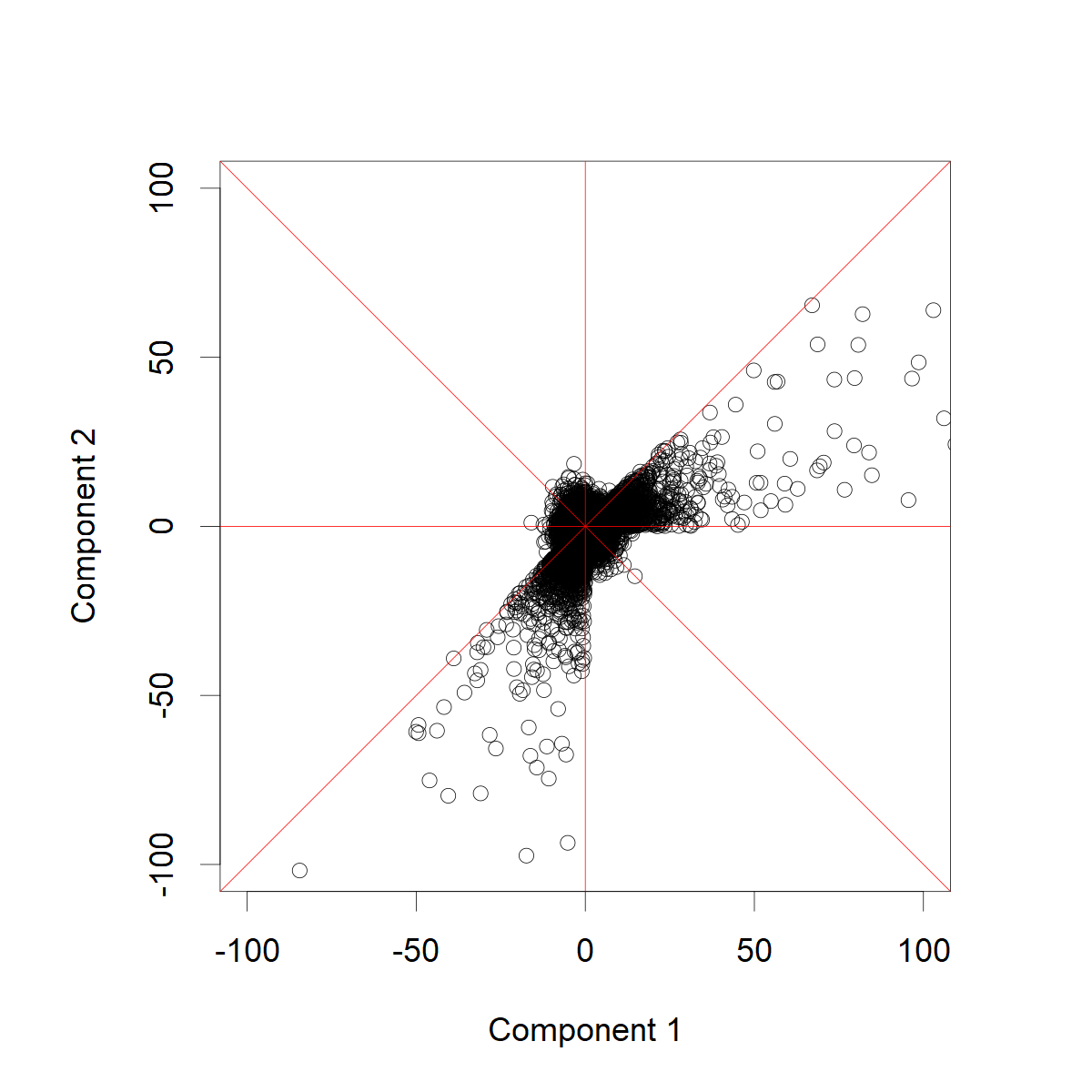}
         \caption{}
         \label{pic_orig}
     \end{subfigure}
     \hfill
     \begin{subfigure}[b]{0.49\textwidth}
         \centering
         \includegraphics[width=\textwidth]{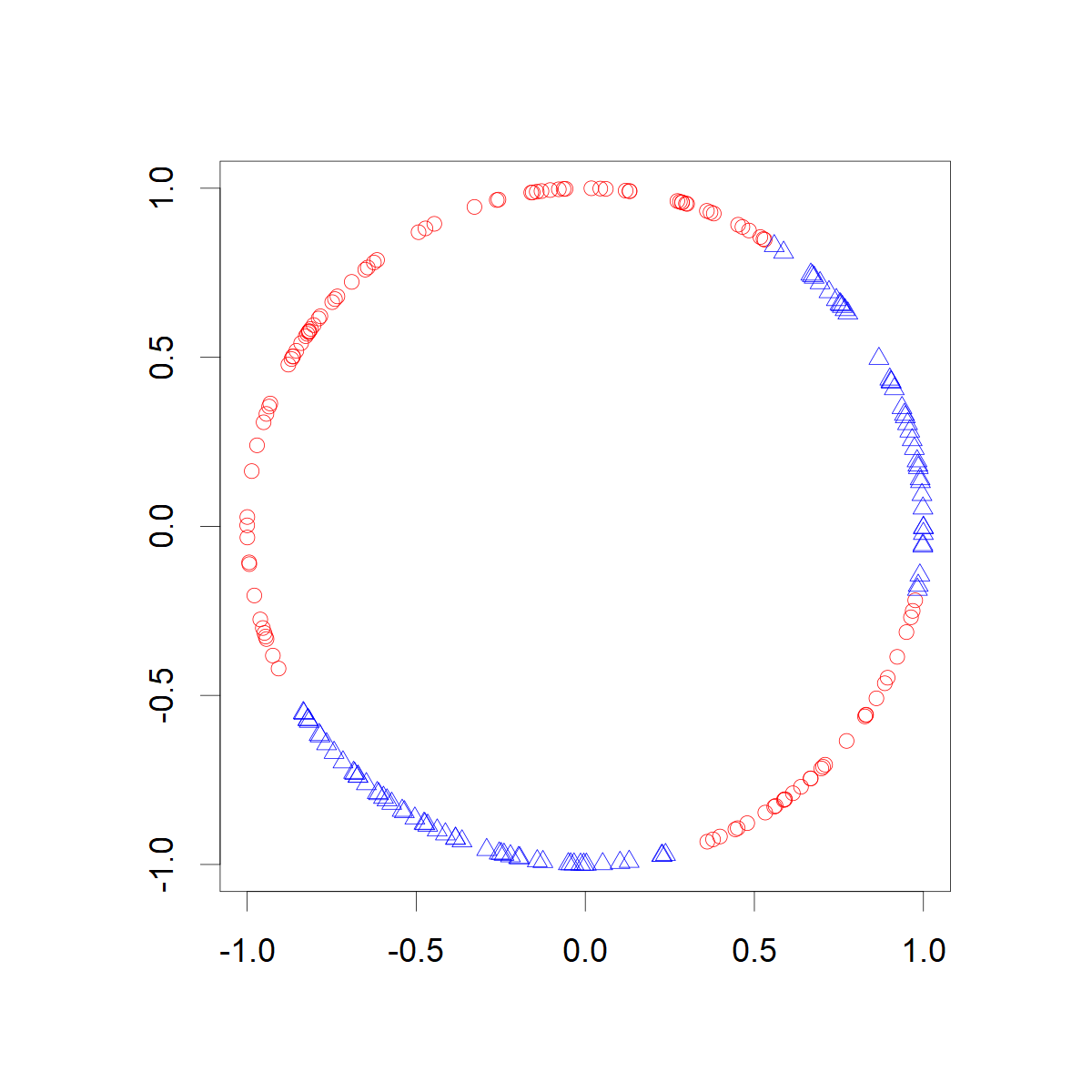}
         \caption{}
         \label{pic_estim}
     \end{subfigure}
        \caption{Illustration of simulated data. The original data set is presented in \ref{pic_orig}. Subfigure \ref{pic_estim} illustrates which directions are accepted to or rejected from the final estimate.}
        \label{fig:three graphs}
\end{figure}

The idea presented in Inequalities \eqref{eq_testineq} and \eqref{eq_testineq2} was studied numerically. The testing sets $D$ were chosen to be open balls of the form $B(\vv,s_\vv)$. The radius $s_\vv$ was selected to be the smallest number such that $10\%$ of all observations had directional components in $B(\vv,s_\vv)$. Figure \ref{pic_estim} presents the tested directions on the unit sphere. Each point corresponds to a fixed value of $\vv$. The red dots are the directions which were rejected, i.e.\ the value $\hat{g}(k,B(\vv,s_\vv))$ of \eqref{eq_g_hattu} is too high given the tolerance $c$ so that \eqref{eq_testineq} holds. The blue triangles are the accepted centers of cones from which the final estimate is formed. In the final estimate, we have removed all open balls that were rejected for some direction $\vv$. The values of the parameters were set to be $c=0.5$ and the top $0.5 \%$ of the observations, in $l_2$ norm, were used.

\begin{figure}[htb]
     \centering
     \includegraphics[width=\textwidth]{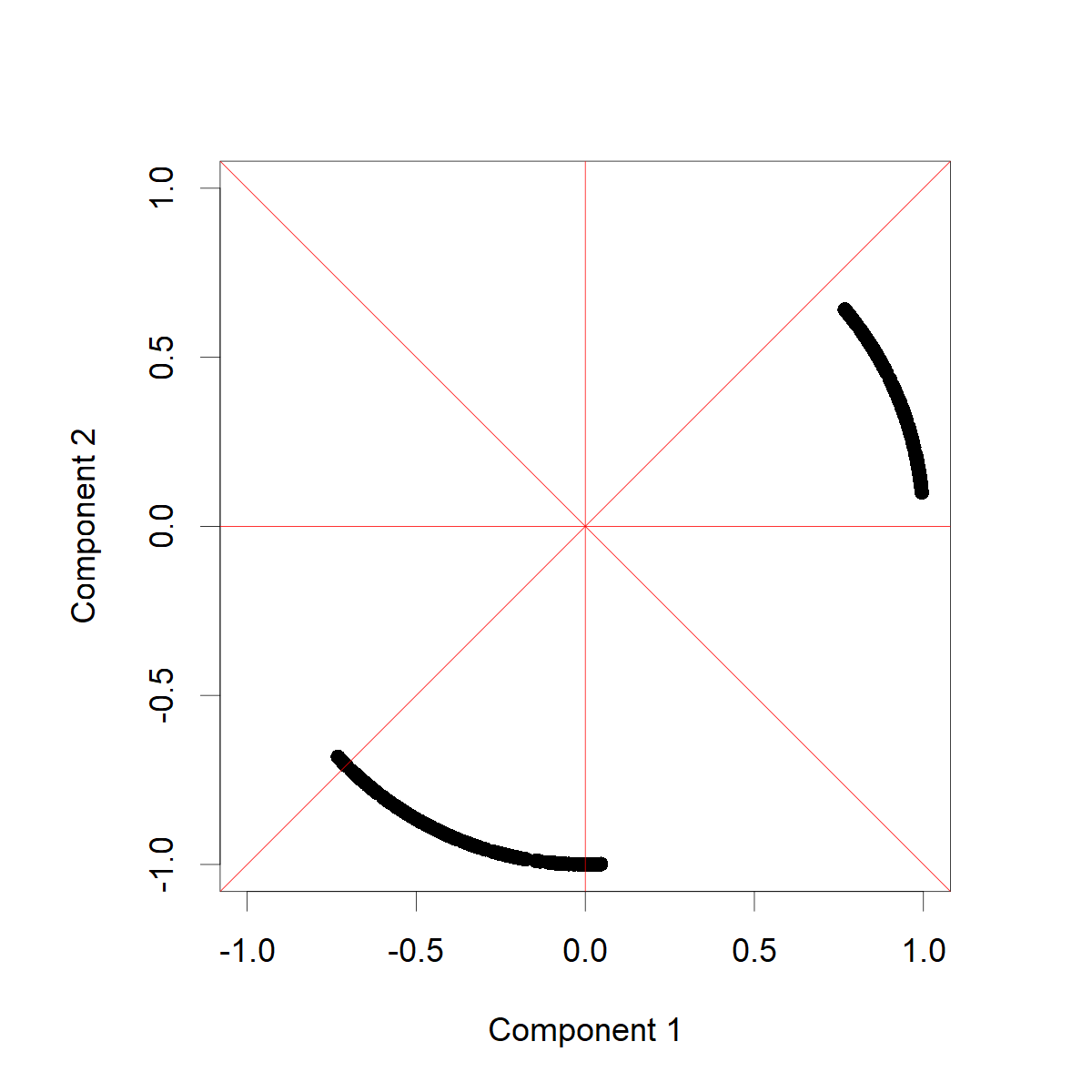}
         \caption{A preliminary estimate for $S$ based on numerical data.}
         \label{pic_results}
\end{figure}

In Figure \ref{pic_results}, the preliminary estimate for $S$ is presented on the unit sphere. The method identifies correctly the heaviest directions. The estimate seems to be most accurate near the centres of the cones and less accurate near the edges between heavier and lighter radial components.

\subsection{On the detection accuracy with Pareto tails}\label{sec_pareto_detection}

We used the same algorithm as in Section \ref{sec_simulation} to analyse a similar data set except that all the radial components have Pareto distribution. The Pareto index is the heaviest in the same cones as earlier. The remaining directions have lighter Pareto tails. 

\begin{figure}[h]
  \begin{subfigure}[b]{0.24\textwidth}
    \includegraphics[width=\textwidth]{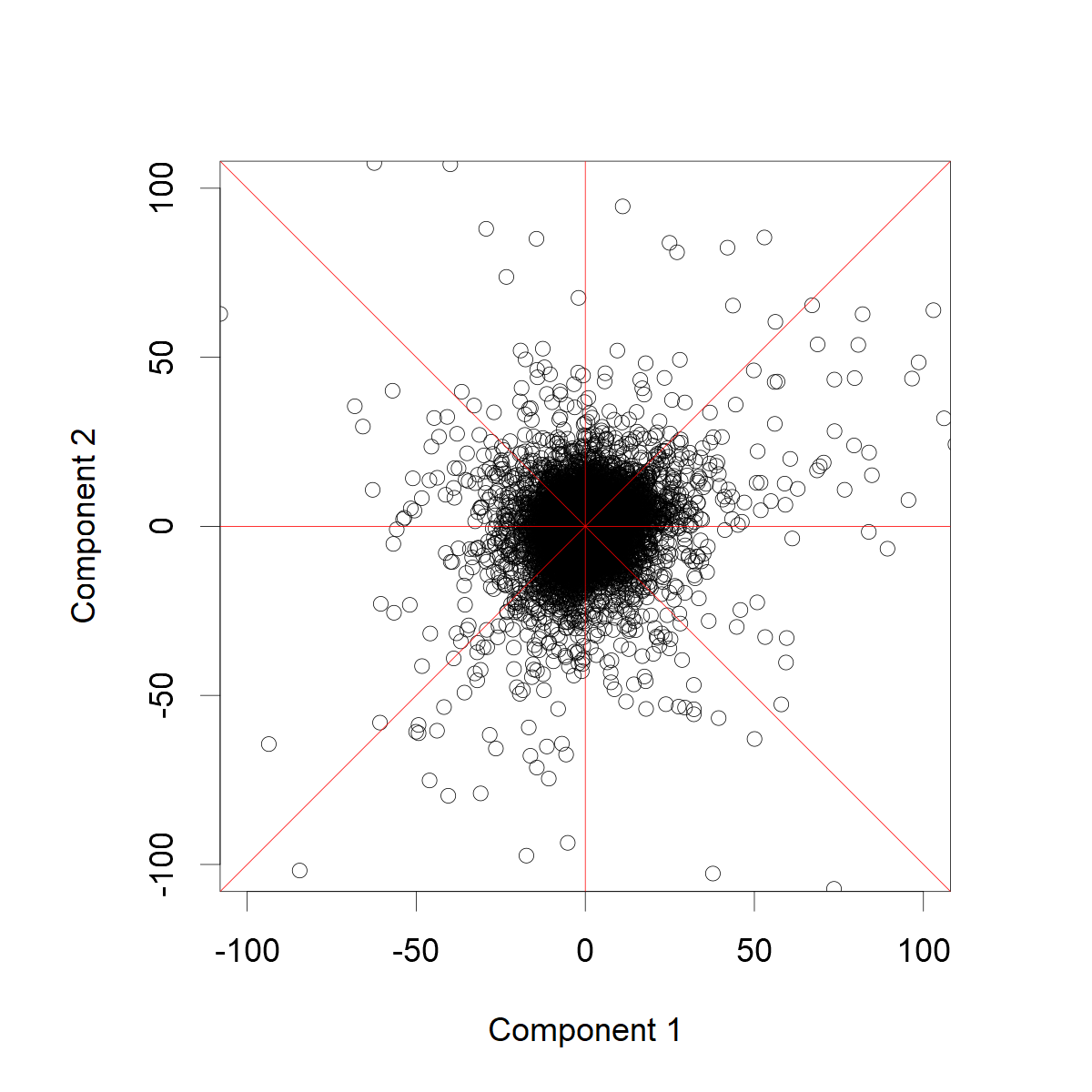}
    \caption{}
    \label{dmulti:1}
  \end{subfigure}
  \begin{subfigure}[b]{0.24\textwidth}
    \includegraphics[width=\textwidth]{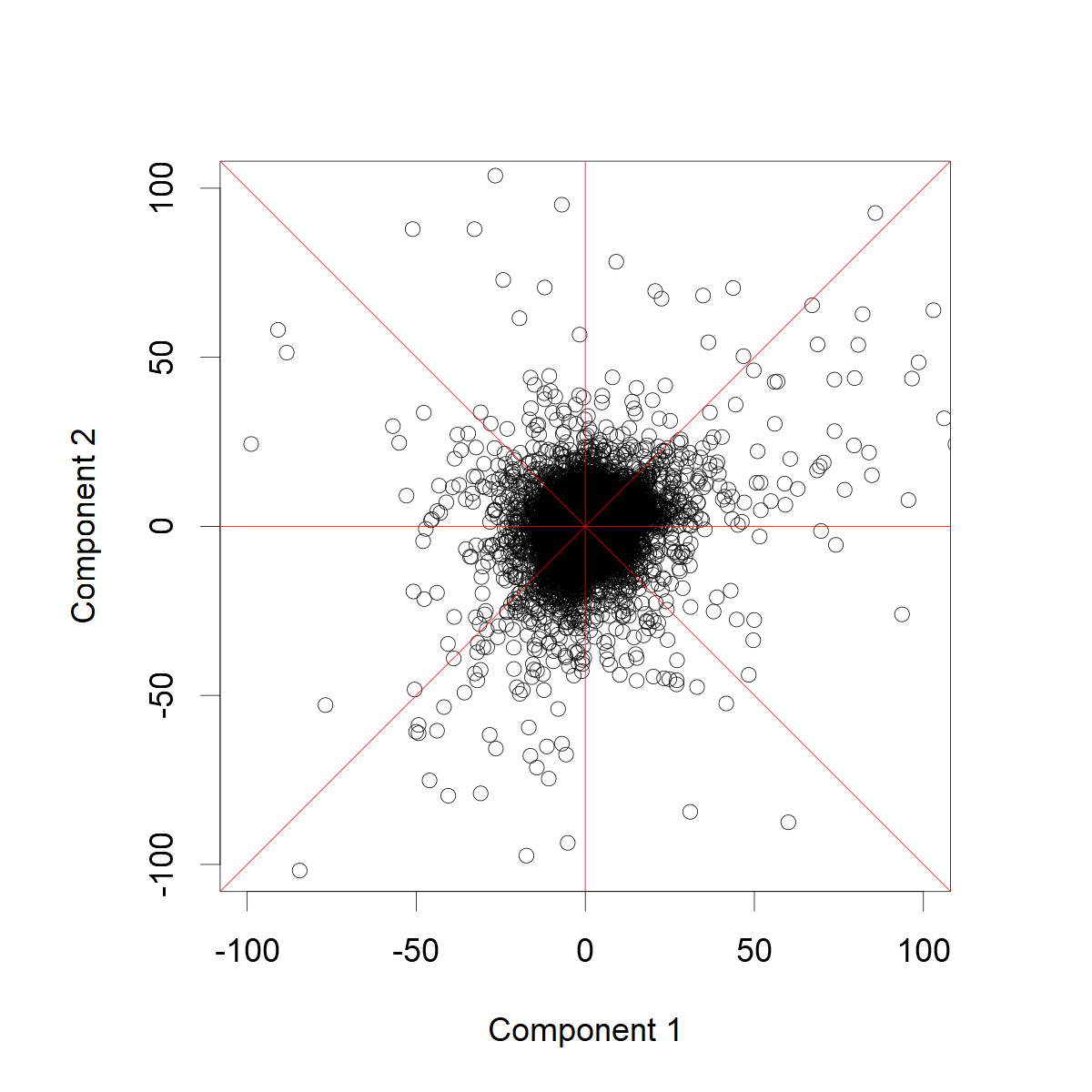}
    \caption{}
    \label{dmulti:2}
  \end{subfigure}
  \begin{subfigure}[b]{0.24\textwidth}
    \includegraphics[width=\textwidth]{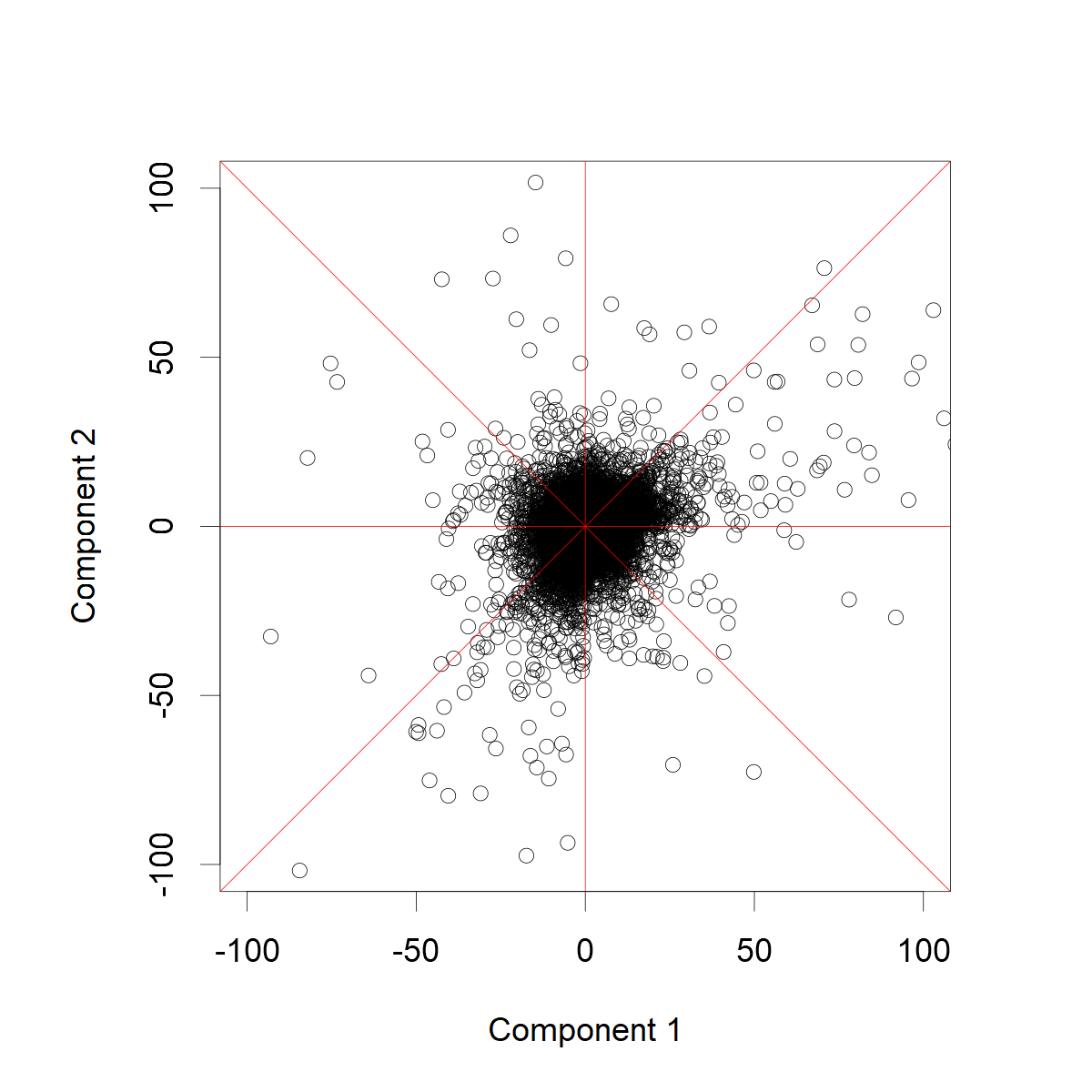}
    \caption{}
    \label{dmulti:3}
  \end{subfigure}
  \begin{subfigure}[b]{0.24\textwidth}
    \includegraphics[width=\textwidth]{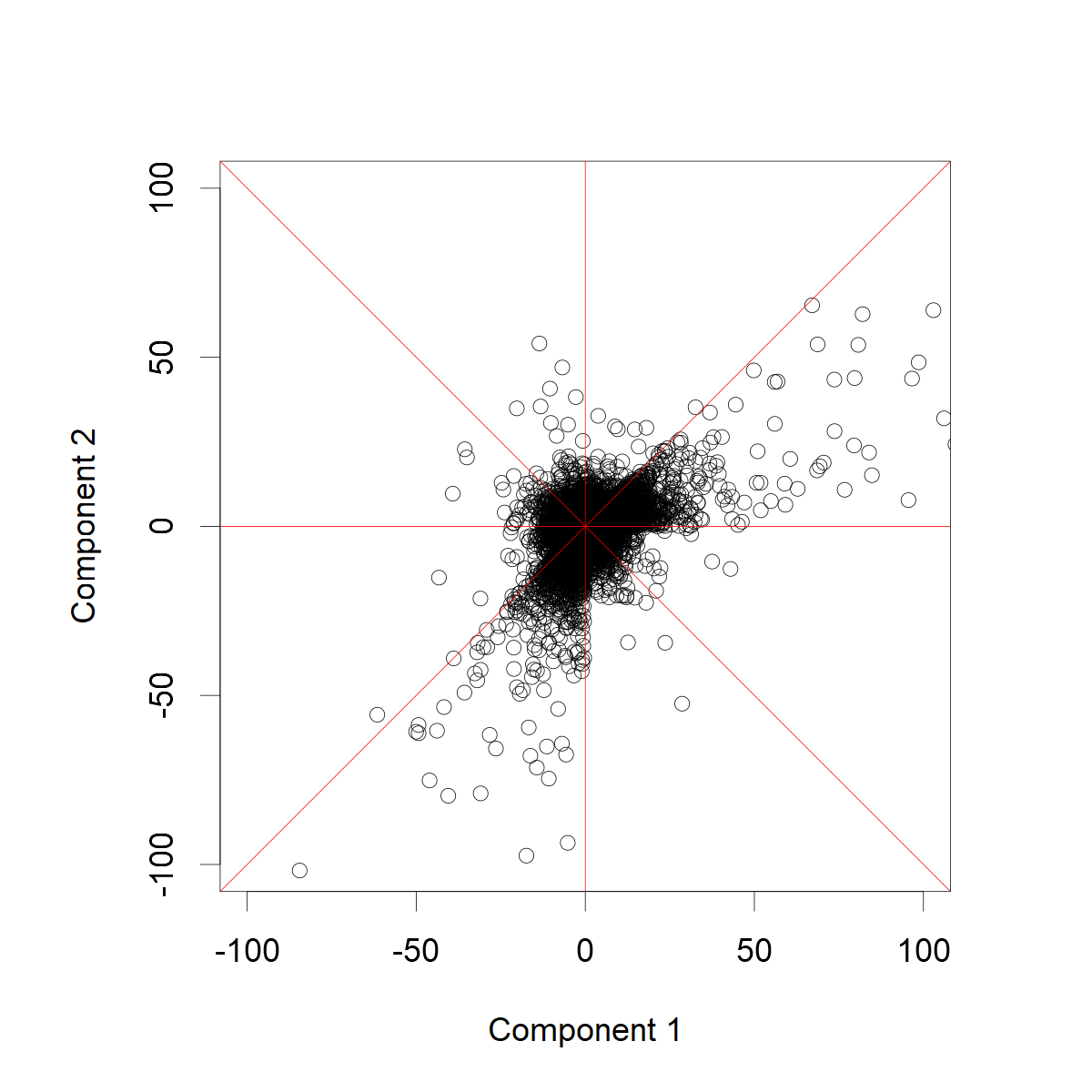}
    \caption{}
    \label{dmulti:4}
  \end{subfigure}
    \caption{Projected original data with Pareto distributed radial components. The heaviest tail has index $2$ and the lighter tails have tail index values $2.3$, $2.4$, $2.5$ and $3$ in subfigures \ref{dmulti:1}, \ref{dmulti:2}, \ref{dmulti:3}, \ref{dmulti:4},  respectively. The same random seed is used in all simulations.}
    \label{fig:danishmulti}
\end{figure}

\begin{figure}[h]
  \begin{subfigure}[b]{0.24\textwidth}
    \includegraphics[width=\textwidth]{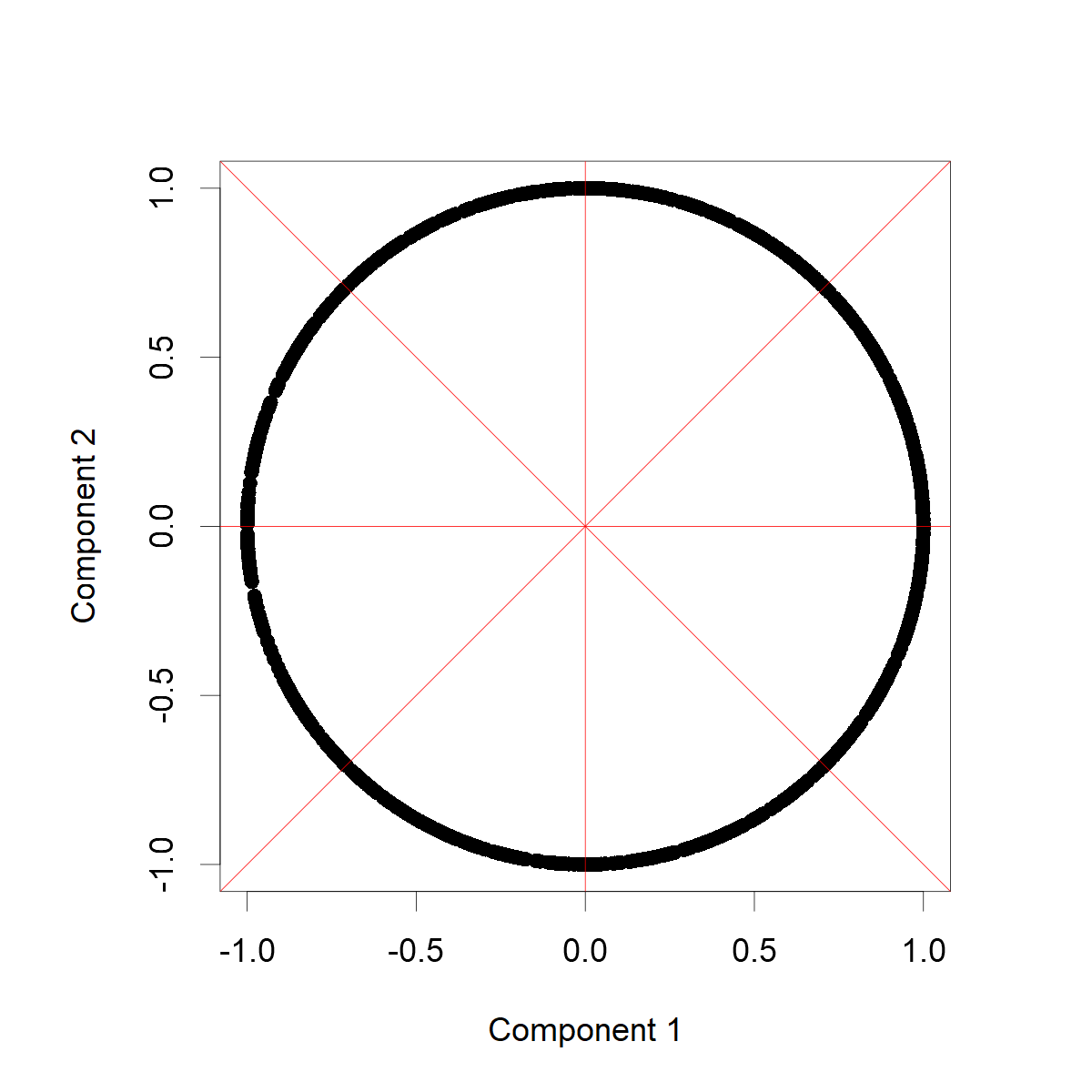}
    \caption{}
    \label{dmulti:12}
  \end{subfigure}
  \begin{subfigure}[b]{0.24\textwidth}
    \includegraphics[width=\textwidth]{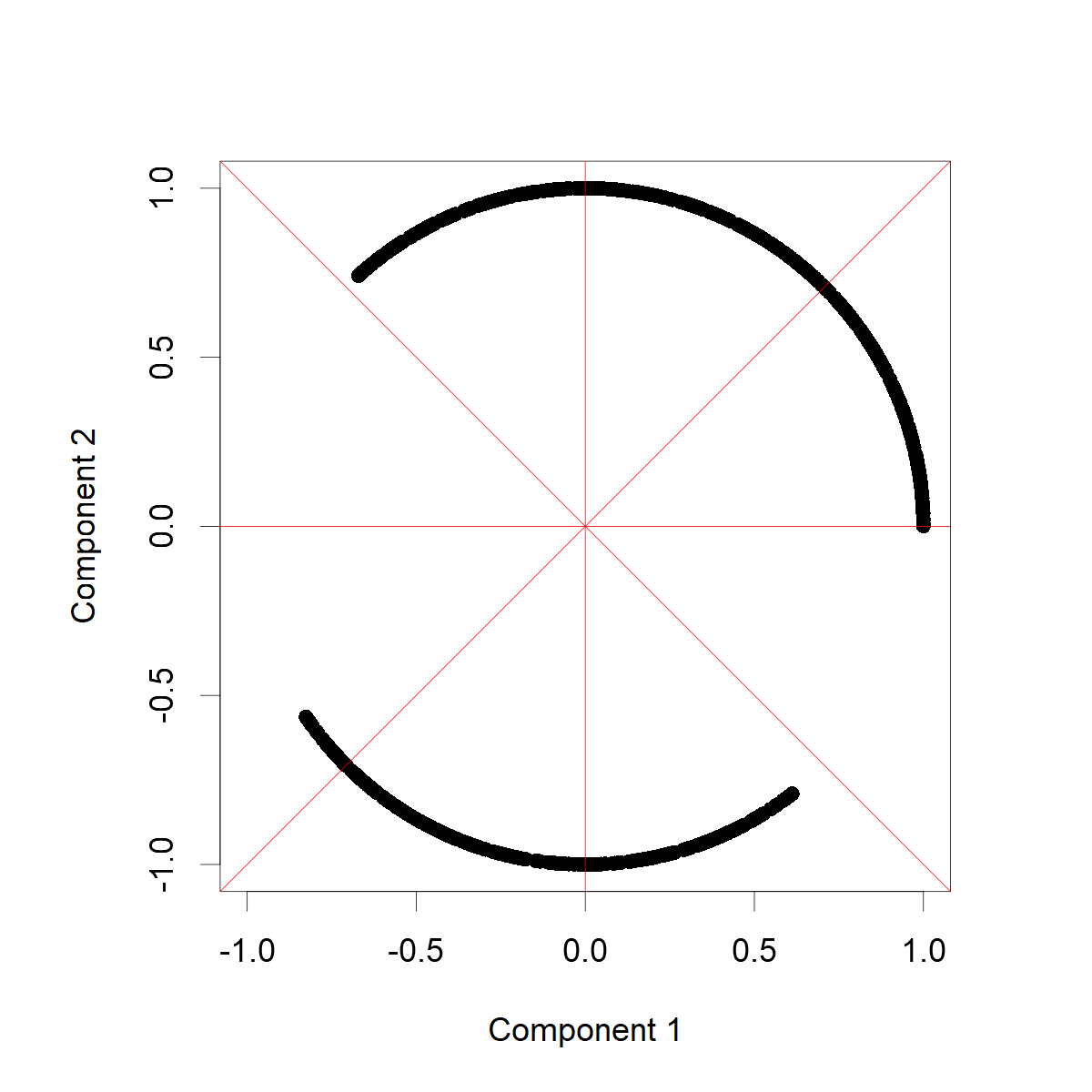}
    \caption{}
    \label{dmulti:22}
  \end{subfigure}
  \begin{subfigure}[b]{0.24\textwidth}
    \includegraphics[width=\textwidth]{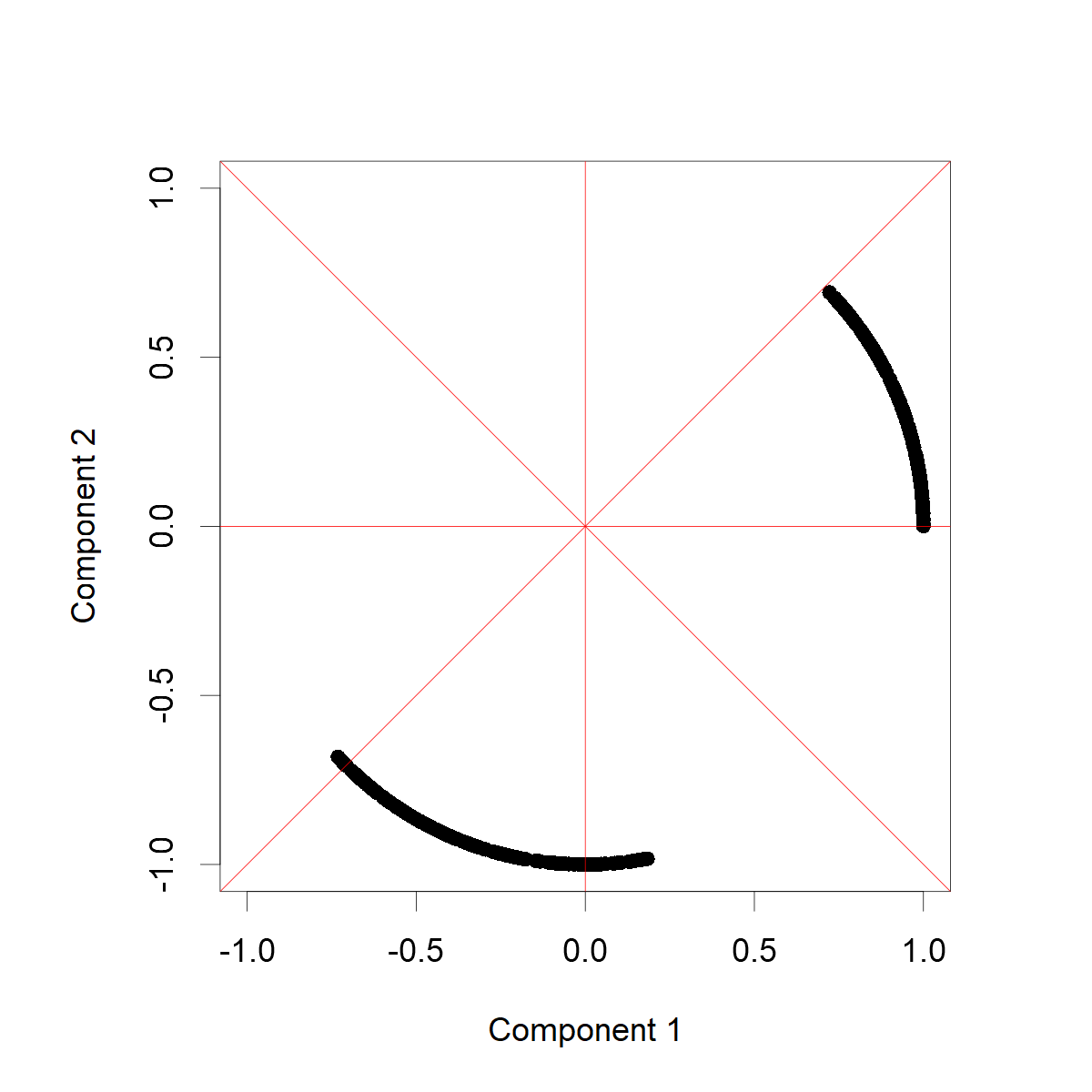}
    \caption{}
    \label{dmulti:32}
  \end{subfigure}
  \begin{subfigure}[b]{0.24\textwidth}
    \includegraphics[width=\textwidth]{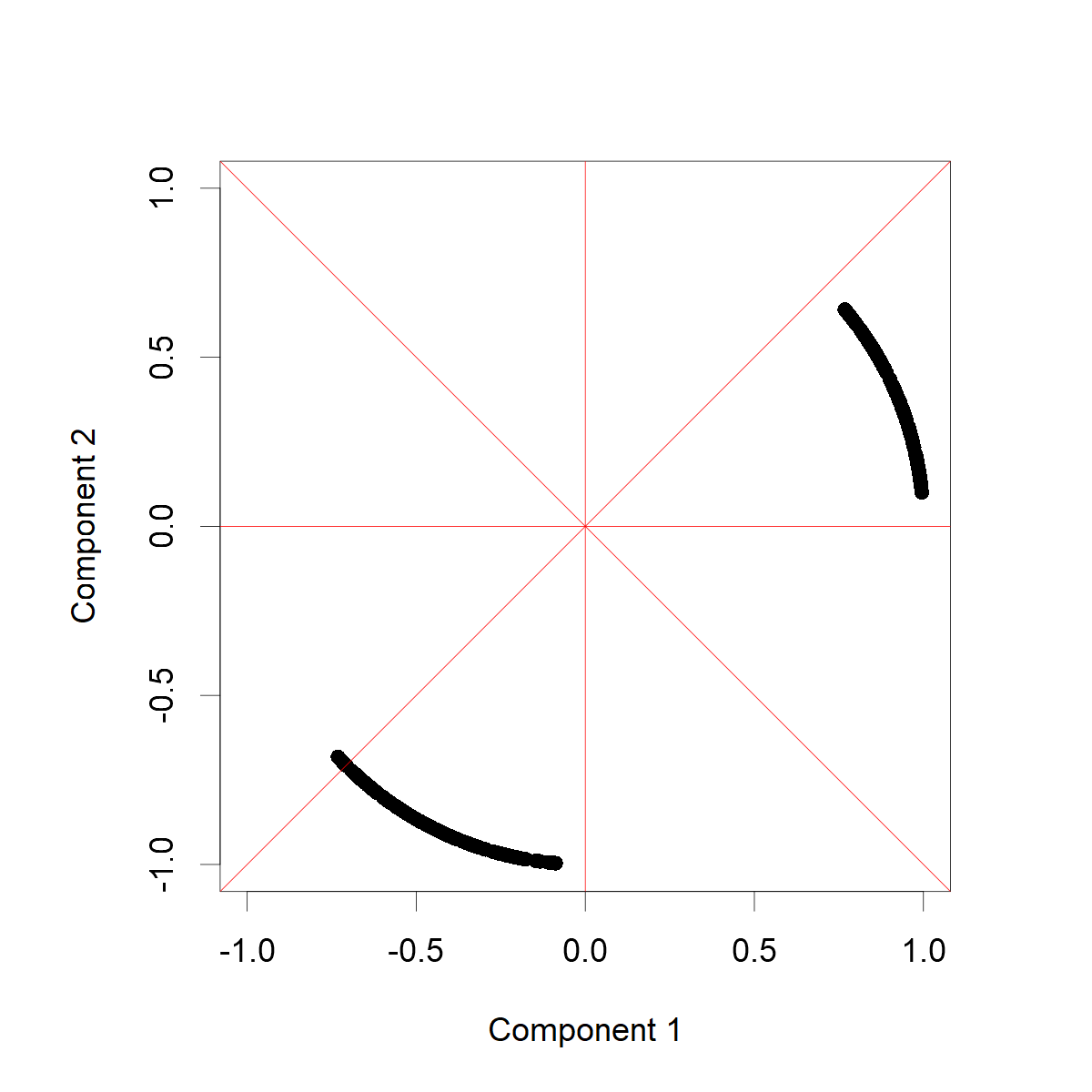}
    \caption{}
    \label{dmulti:42}
  \end{subfigure}
    \caption{The estimates for the riskiest directions are presented in Figures \ref{dmulti:12}, \ref{dmulti:22}, \ref{dmulti:32} and \ref{dmulti:42}. When the lighter tail parameter is very close to the value of the heavier parameter $2$, the algorithm cannot recover the set of riskiest directions with the given sample size and selected threshold $c$. When the lighter tail parameter is $2.5$ or larger the produced estimate is quite accurate.}
    \label{fig:danishmulti2}
\end{figure}

\subsection{Example with real data}

We use the algorithm with the same parameters as in Sections \ref{sec_simulation} and \ref{sec_pareto_detection} to study an actual data set. The data contains the daily changes in the prices of gold and silver over a time period ranging from December 3, 1973 to January
15, 2014. It is the same data set that was used in Section 4.5 of \cite{Lehtomaa3} and, consequently, the same modifications to the raw data were made. In particular, we used the logarithmic differences of daily prices in order to get a sequence of two-dimensional vectors that are approximately independent and identically distributed when we study the largest changes. 

We studied only the negative changes in daily prices, i.e.\ when both the prices of silver and gold declined. This decision was made in order to make the results comparable with the earlier results. In fact, the data set studied here is the data set pictured in Figure 9a of \cite{Lehtomaa3} except is contains a few more observations.

\begin{figure}[htb]
     \centering
     \begin{subfigure}[b]{0.49\textwidth}
         \centering
         \includegraphics[width=\textwidth]{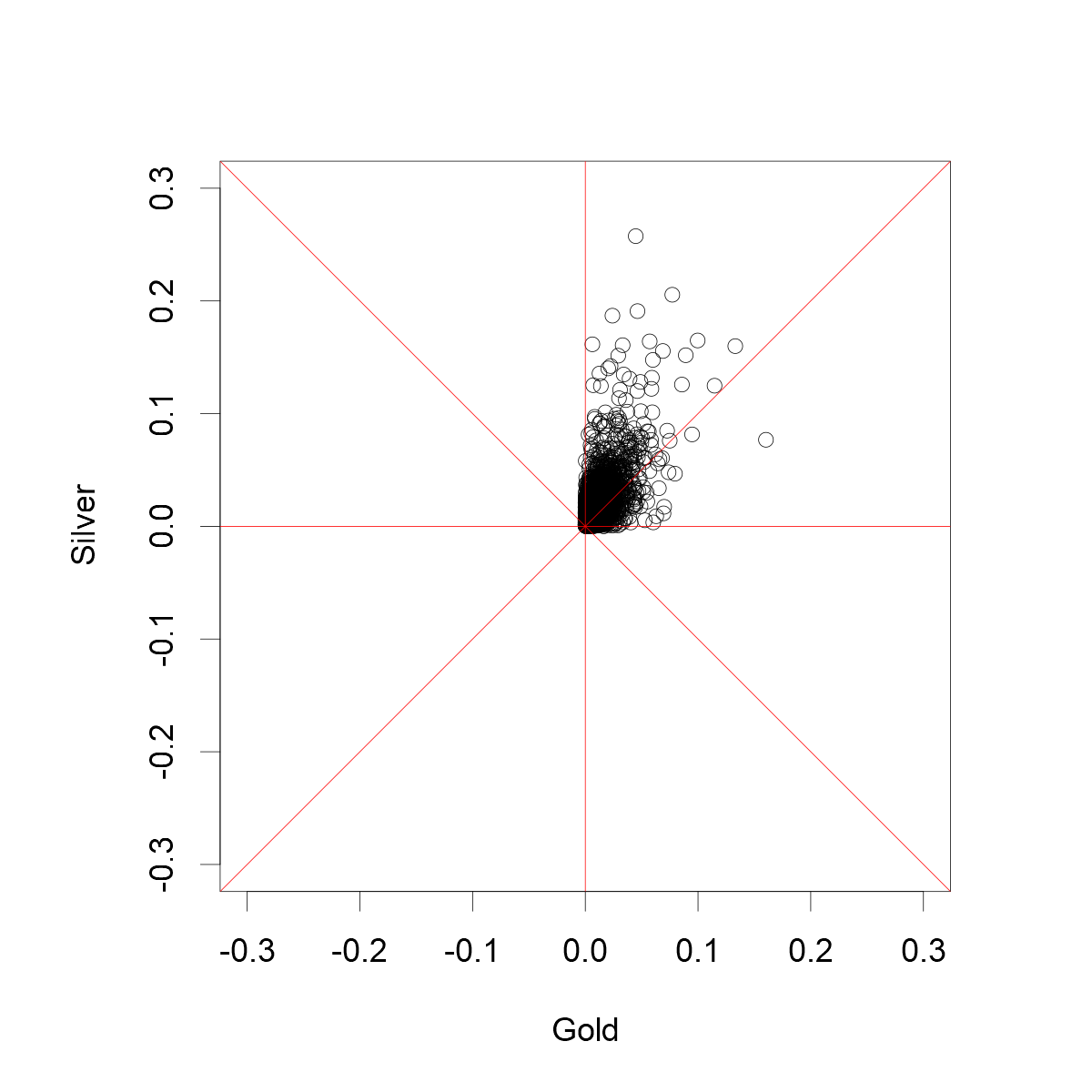}
         \caption{}
         \label{pic_results_real_a}
     \end{subfigure}
     \hfill
     \begin{subfigure}[b]{0.49\textwidth}
         \centering
         \includegraphics[width=\textwidth]{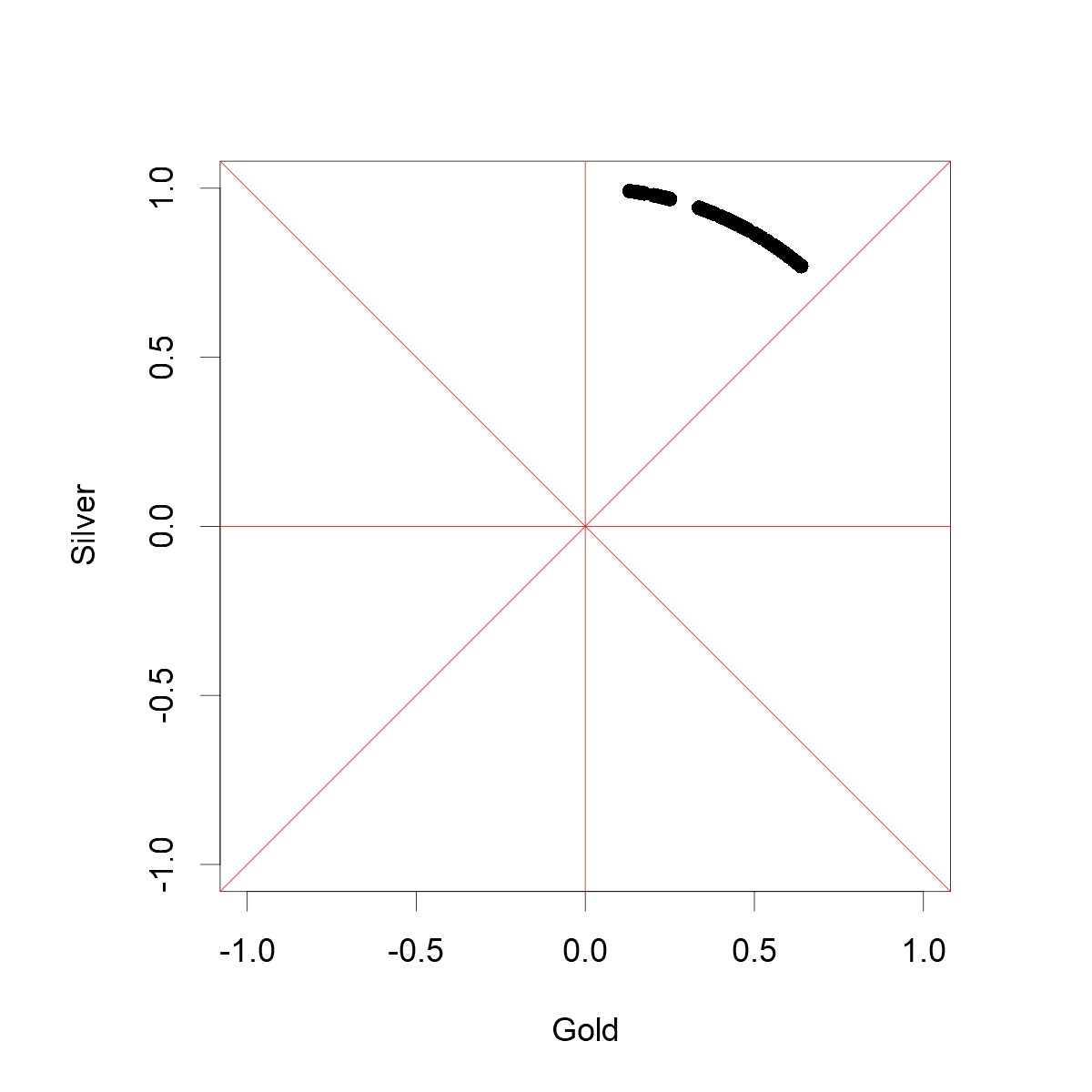}
         \caption{}
         \label{pic_results_real_b}
     \end{subfigure}
        \caption{A preliminary estimate for $S$ in a real data set based on daily changes in the prices of gold and silver is presented on the right. The left picture is a plot of the original data set.}
        \label{pic_results_real}
\end{figure}

Figure \ref{pic_results_real_b} shows the estimate for the riskiest directions. The estimate is consistent with the earlier result obtained in \cite{Lehtomaa3} in the sense that the riskiest observations appear to concentrate on a cone and riskiest observations are more concentrated above the diagonal than below it. It should be noted that the analysis here is performed using the euclidean distance while the analysis of \cite{Lehtomaa3} used the $l_1$ distance and diamond plots. 

In conclusion, a result that is consistent with the earlier study was obtained but without the need to verify the assumptions of the multivariate regularly varying distributions.

\subsection{On the interpretation of estimates}
In practice, we only have access to a finite amount of data. In addition, the user must select suitable values for the parameters $c$ and $k$ in \eqref{eq_testineq} and \eqref{eq_testineq2}. These seem to be the main challenges in accurate detection of the directional components with heaviest tails. Typically, the parameter values are found by experimentation. One can fix $k$ and increase the tolerance until some directions are accepted into the final estimate of $S$. 

We say that tail function $\overline{F_1}$ is heavier than $\overline{F_2}$ if there exists a number $x_0$ such that 
\begin{equation}\label{eq_heavier_tail}
   \overline{F_1}(x)>\overline{F_2}(x) 
\end{equation}
for all $x\geq x_0$. The problem with real data is that the number $x_0$ can be very large. Consequently, the largest observed data points might not be produced by the heaviest tails if the size of the data set is not sufficiently large. For example, lognormal distribution has heavier tail than Weibull distribution with parameter $\beta\in (0,1)$. If $\beta$ is close to $0$, a typical i.i.d.\ sample from these distributions could produce data where the points from Weibull distribution appear to be larger. In the multidimensional setting, the direction can affect the heaviness of the $R$-variable. For this reason, the interpretation of the estimates is the following. \emph{The estimator recovers the heaviest directional components with respect to the size of the data set. It does not exclude the possibility that there exist even heavier directional components than what is detected and which remain undetected due to the limited number of the data points.} To summarise, the estimator detects in which directions the heaviness comparable to the one-dimensional data produced by the normed observations is obtained.

\subsection{Risk ranking} 

Given an accurate estimate on $S$, we can form a new data set where the directions corresponding to the estimated $S$ have been removed. Given that the data satisfies the necessary assumptions, we can rerun the estimation to find the "next riskiest" directions. We can continue the process of deleting the estimated riskiest set even more than once. This could be viewed as finding the ranking of the riskiest sets on the unit sphere.  

In this process, it is possible that the second riskiest directions correspond to tails that do not belong to the same distribution family as the original set of the riskiest directions. For example, the tail of $R$ could have power law in the set of the riskiest directions and, say, a lognormal tail in the set of the next riskiest directions. 

If the random vector $\X$ has a multivariate regularly varying distribution, the aim is typically to estimate the support of the angular measure \cite{Resnick2}. In this case, the method described in Section \ref{sec_estimator} detects the set of riskiest directions which correspond to the tail index of the distribution. If the set $S$ is removed, the remaining set could still have power law behaviour in the sense of hidden regular variation \cite{Resnick2}. In this case, the preliminary empirical method suggested in Section \ref{sec_estimator} could reveal the set of the next riskiest directions which correspond to the tail index of the hidden multivariate regularly varying component.

\section{Conclusions}

It seems plausible that the theoretical result in Theorem \ref{thm_minimalset} could be used to create statistical estimators for the set $S$. For example, given an estimate for the set of riskiest directions, an insurance company with heavy-tailed total losses could find the root cause of heavy-tailedness. Here, the cause is found by identifying individual components or interaction between multiple components that produce heavy-tailedness to the total loss. If the company understands the set of riskiest directions, it is possible to formulate hedging or reinsurance strategies that potentially mitigate the largest risks. In this sense, understanding the set of riskiest directions tells why the entire vector is heavy-tailed and offers a strategy for reducing risk.

The presented results offer a starting point for creating rigorous statistical estimators with a clear work-flow that can be implemented as a computer algorithm. At first, one checks for the heavy-tailedness of the observations by calculating the empirical hazard function of the normed observations which produce a one-dimensional data set. Once the heaviness of this one-dimensional data has been established and the empirical hazard function turns out to be concave, we can search the directions where heaviness of the observations corresponds to the heaviness of the one-dimensional data. A way to implement the method is to study cones around fixed points and determine the size of each cone based on a given portion of the total observations, e.g.\ we can find the smallest cone that contains 10\% of the observations as in the presented examples. This avoids the problem encountered with grid-based methods where the space is divided into cells of equal size and where it is possible that some of the cells remain sparsely populated by the observations.

As the examples with simulated and real data show, an algorithm can be also implemented in the case where there are directions with only few observations or no observations at all.  Furthermore, the data does not have to be transformed or pre-processed before applying the method but it gives an idea where the riskiest directions are also in the case where, for instance, some components have a different Pareto indices than other components. 

Detecting small differences requires large amount of data. To us, this means that there must be more data if the tails associated with different directions are almost equally heavy. The presented idea works best if there exists one or more directions where the tail is substantially heavier than in other directions. As in earlier models, the practical application of the method requires some parameters to be set by the user. In particular, choosing the threshold $k$ is not easy, but this is a well-known problem which exists in different forms in most heavy-tailed modeling \cite{Nguyen1}. 

\section*{Acknowledgement}
Suggestions made by the reviewers helped to improve the manuscript.

\bibliographystyle{abbrv}
\bibliography{kirjallisuus1}

\end{document}